\documentclass[a4paper]{amsart}

\usepackage{dsfont}
\usepackage{amsmath, amsthm, amssymb, amsfonts,mathrsfs,comment}
\usepackage{enumerate}
\usepackage{upgreek}
\usepackage{mathabx}
\usepackage{graphicx}
\usepackage{hyperref}

\numberwithin{equation}{section}
\allowdisplaybreaks
\theoremstyle{plain}
\newtheorem{thm}{Theorem}
\newtheorem{Def}{Definition}
\newtheorem{lem}{Lemma}

\newtheorem{cor}{Corollary}

\newtheorem{rem}{Remark}
\newtheorem{Ass}{Assumptions}
\newtheorem*{def*}{Definition}
\newtheorem*{thm*}{Theorem}
\newtheorem{pipi}[thm]{Proposition}

\renewcommand{\epsilon}{\varepsilon}
\renewcommand{\theta}{\vartheta}
\renewcommand{\phi}{\varphi}
\renewcommand{\pi}{\uppi}
\renewcommand{\delta}{\updelta}

\DeclareMathOperator{\fff}{F}
\DeclareMathOperator{\mmm}{M}
\DeclareMathOperator{\ttt}{T}

\DeclareMathOperator{\gagaga}{STFT}
\DeclareMathOperator{\wawawa}{W}

\DeclareMathOperator{\ststst}{S}

\DeclareMathOperator{\Span}{span}

\newcommand{\pt}[1]{\left( #1 \right) }
\newcommand{\ptg}[1]{\left\{ #1 \right\} }
\newcommand{\ptq}[1]{\left[ #1 \right] }

\newcommand{\abs}[1]{\left\lvert #1\right\rvert}
\newcommand{\norm}[1]{\left\|  #1  \right\| }

\newcommand{\Gtp}[1]{ \pt{\gagaga_\phi #1}}
\newcommand{\Stu}[1]{\ststst #1  }

\newcommand{\Wa}[2]{\wawawa_{#1 } #2 }
\newcommand{\Wap}[2]{\pt{\wawawa_{#1 } #2} }

\def\Rr{\mathbb R}
\def\Raggio{R}
\def\R{\mathbb R}

\def\Z{{\mathbb Z}}
\def\N{{\mathbb N}}

\def\L{L^1 \pt{\Rr}}

\def\Ll{L^2 \pt{\Rr}}

\def\F{\fff}
\def\Fi{\fff^{-1}}
\def\M{\mmm}

\def\T{\ttt}

\def\ephipt{E^\phi_{p,\tau}}

\def\Fphipt{F^\phi_{p,\tau}}
\def\Fphiptp{F^\phi_{p',\tau'}}
\def\Si{\mathscr{S}}

\def\dnbt{D_{p,\tau}}
\def\dnbtp{D_{p',\tau'}}
\def\dnbtneg{D_{-p,\tau}}
\def\dnbti{D_{p',\tau}}

\def\dnbto{D_{p,0}}
\def\fdnbt{f_{p,\tau}}
\def\cpfj{c^{\phi}_{p,j}}

\def\bxi{\pt{b,\xi}}

\def\ba{\pt{b,a}}

\def\Lloi{L^2\pt{[0,1]}}
\def\cchi{\widecheck{\chi}}
\def\ecchipt{E^{\widecheck{\chi}}_{p,\tau}}
\def\pii{2\pi \mathrm{i}\,}

\begin{document}

\begin{abstract}
Since its appearing in 1996, the Stockwell transform ($\ststst$-transform) 
has been applied to medical imaging, geophysics
and signal processing in general. 
In this paper, we prove that the system of functions (so-called DOST basis) is indeed an orthonormal basis of
$\Lloi$, which is time-frequency localized, in the sense of Donoho-Stark Theorem (1989).  
Our approach provides a unified setting in which to study the Stockwell 
transform (associated to different admissible windows) and its orthogonal decomposition.
Finally, we introduce a fast -- $\mathcal{O}\pt{N  \log N }$ -- algorithm  to compute the Stockwell coefficients for
an admissible window. Our algorithm extends the one proposed by Y. Wang and J. Orchard (2009).
\end{abstract}
\vspace{1em}
\title{Window-dependent Bases for Efficient Representations of the Stockwell Transform}

\maketitle

\noindent \begin{center}
\footnotesize{\textbf{U. Battisti}}\\
\footnotesize{Dipartimento di Matematica,\\ Universit\`{a} degli Studi di Torino,\\ via
Carlo Alberto 10,\\ 10123 Torino, Italy}
\par\end{center}

\noindent \begin{center}
\footnotesize{\textbf{L. Riba}}\\
\footnotesize{Dipartimento di Matematica,\\ Universit\`{a} degli Studi di Torino,\\ via
Carlo Alberto 10,\\ 10123 Torino, Italy}
\end{center}
\vspace{2em}

\noindent \textbf{\small{Keywords and phrases: }}{\small{Stockwell transform, Orthonormal basis, 
Short-time Fourier transform, Wavelet transform. }}{\small \par}
\vspace{1em}


\section{Introduction}
Let $f$ be a  signal with finite energy, that is $f \in \Ll$,  and let $\phi$ be a window in $\Ll$. 
Then, following M. W. Wong and H. Zhu \cite{WZ07},  
we define the Stockwell transform ($\ststst$-transform) $\Stu_{\phi}{f}$ as 
\begin{align}
\pt{\Stu_{\phi}{f}}\pt{b,\xi} & =\pt{2\pi}^{-1/2}\int_{\Rr}
e^{-\pii t\xi}\,f\pt{t}\abs{\xi} \overline{\phi\pt{\xi\pt{t-b}}}\, dt,\qquad b,\xi\in\Rr.\label{eq:cont_stoc_time}
\end{align}
It is possible to rewrite the $\ststst$-transform with respect to the Fourier transform of the analyzed signal:
\begin{align}
\pt{\Stu_{\phi}{f}}\pt{b,\xi} & =\int_{\Rr}e^{\pii b\zeta}\,
\widehat{f} \pt{\zeta+\xi}\overline{\widehat{\phi}\pt{\frac{\zeta}{\xi}}}\, d\zeta,\qquad b,\xi\in\Rr, \quad \xi\neq 0 ,\label{eq:cont_stoc_freq}
\end{align}
where $ \widehat{f} $ is the Fourier transform of the signal $ f$, given by
\begin{align*}
  \widehat{f}\pt{\xi}& =\pt{\F f}\pt{\xi} = \pt{2\pi}^{-1/2}\int_{\Rr}e^{-\pii t\xi}\,f(t)dt,\qquad \xi\in\Rr .
\end{align*}
We fix some notation: we denote with $\widecheck{f}$  or $\Fi f $ the inverse Fourier
transform of a signal $f$. $\N=\ptg{0,1, \ldots}$ is the set of non negative integers,
$\Z=\ptg{\ldots, -1,0,1,\ldots, }$ is the set of integers.

The $\ststst$-transform was initially defined by R. G. Stockwell, L. Mansinha and R. P. Lowe in \cite{stockwell1996localization} using a Gaussian window 
\begin{align*}
g\pt{t} = e^{-t^2 /2},\qquad t\in\Rr . 
\end{align*}  
In this case,
\begin{align}
\pt{\Stu_{g}{f}}\pt{b,\xi} & =
\pt{2\pi}^{-1/2}\int_{\Rr}e^{-\pii t\xi}\,f\pt{t} \abs{\xi} e^{- \pt{t-b}^{2} \xi^{2}/2}\, dt,\qquad b,\xi\in\Rr,\label{eq:cont_stoc_time_gaussian}
\end{align}
which, in the alternative formulation, becomes
\begin{align}
\pt{\Stu_{g}{f}}\pt{b,\xi} & =\int_{\Rr}e^{\pii \zeta b}\, \widehat{ f}\pt{\zeta+\xi} e^{-2\pi^{2}\zeta^{2}/\xi^{2}}\, 
d\zeta,\qquad b,\xi\in\Rr,\quad \xi\neq 0.\label{eq:cont_Stock_freq_gaussian}
\end{align}
The natural discretization of \eqref{eq:cont_Stock_freq_gaussian}, introduced in \cite{stockwell1996localization},  is given by
\begin{align}
\label{eq:disc_stoc_redundant}
\pt{\Stu_{g}{f}}\pt{j,n} & =
\sum_{m=0}^{N-1}e^{\pii mj /N}\widehat{ f}\pt{m+n} e^{- 2\pi^{2}m^{2} /n^2},
\end{align}
where $j=0, \ldots, N-1$ and $n=1, \ldots, N-1$. For $n=0$, it is set
\begin{align*}
 \pt{\Stu_{g}{f}}\pt{j,0}= \frac{1}{N}\sum_{k=0}^{N-1} f(k), \qquad j=0, \ldots, N-1.
\end{align*}
In the literature, (\ref{eq:disc_stoc_redundant}) is called redundant (discrete) Stockwell transform.
Unfortunately, the redundant Stockwell transform has a high computational cost  -- $\mathcal{O}\pt{N^2 \log N}$. 
To overcome this problem, R. G. Stockwell  introduced in \cite{ST07}, without a mathematical proof, 
a basis for periodic signals with finite energy, \emph{i.e.}  $\Lloi$,  given by
\begin{align}
\label{dostbasis}\bigcup_{p\in\Z} D_p & = \bigcup_{p\in\Z} \ptg{\dnbt}_{\tau = 0}^{\beta\pt{p}-1}.
\end{align}
This basis, precisely defined in Section \ref{sec:propbase}, is adapted to octave samples in the frequency domain. 
The decomposition of a periodic signal $f$ in this basis is called in the literature the discrete orthonormal Stockwell transform (DOST). 
The related coefficients
\begin{align*}
 \fdnbt = \pt{f, \dnbt}_{\Lloi}, 
\end{align*}
are called DOST coefficients.

In this paper we prove that this basis is not suited  to the standard $\ststst$-transform with Gaussian window \eqref{eq:cont_stoc_time}, 
rather to an $\ststst$-transform associated with a characteristic function (boxcar window). This fact was already pointed out by R. G. Stockwell himself
in \cite{ST07} and \cite{ST07b}.
The computational complexity of the algorithm suggested by R. G. Stockwell was still high: $\mathcal{O}(N^2 )$. In 
2009, Y. Wang and J. Orchard \cite{WO09} proposed a fast algorithm which reduces drastically the complexity to $\mathcal{O}(N \log N)$; the same complexity of the FFT. 
This achievement allowed a wider application of
the $\ststst$-transform to image analysis.

We provide an adapted basis of $\Lloi$ on which to decompose the Stockwell transform with a 
general admissible window $\phi$. 
Assume that we can find such a basis $E_{p}^{\phi}$ of $\Lloi$,  
depending on the choice of $\phi$. Then, by linearity, we can write
\begin{align}
  \label{eq:framedual}
  \pt{\Stu_{\phi}{f}}\pt{b,\xi} & =
  \sum_{p} c_p^\phi \,\pt{\Stu_{\phi}{E_{p}^{\phi}}}\bxi
\end{align}
where 
\begin{align*} 
f=\sum_j c_p^\phi E_p^{\phi}.
\end{align*} 
An ideal basis would satisfy the following properties:
\begin{enumerate}[(i)]
\item \label{enu:proport} $E_p^\phi$ is an orthonormal basis of $\Lloi$, so that 
\begin{align*}
c_p^\phi=\pt{f, E_p^\phi}_{\Lloi}; 
\end{align*}
\item $\pt{\Stu_{\phi}{E_{p}^{\phi}}}\pt{b,\xi}$ is \emph{local} in time;\label{enu:primaprop}
\item $\pt{\Stu_{\phi}{E_{p}^{\phi}}}\pt{b,\xi}$ is \emph{local} in frequency;\label{enu:secondaprop}
\item we can find a fast algorithm -- $\mathcal{O}\pt{N \log N}$ -- 
to compute the coefficients \label{enu:terzaprop} 
\begin{align*}
\pt{f,E_p^\phi}_{\Lloi}.
\end{align*}
\end{enumerate}
We prove that \eqref{dostbasis} is indeed
an orthonormal basis of $\Lloi$ satisfying conditions \eqref{enu:proport}, \eqref{enu:primaprop}, \eqref{enu:secondaprop} and
\eqref{enu:terzaprop} if $\phi=\cchi=\F^{-1}\chi_{\pt{-\frac{1}{3}, \frac{1}{3}}}$.
In particular, we prove that\footnote{See Remark \ref{rem:dost} for the precise statement.}
\[
  E_{p, \tau}^{\cchi}=\dnbt .
\]
Moreover, in Proposition \ref{prop:DOST} we clarify the connection
between the Stockwell coefficients and the value of the $S$-transform
with window $\cchi$.

Let $\phi$ be an admissible window, we introduce the  basis 
\begin{align}
  \label{eq:baseephip}
\ephipt \, ,
\end{align}
such that\footnote{Equality \eqref{eq:asa} must be interpreted with care, we refer to Section
\ref{sec:baseephi} for the precise statement.} 
\begin{align}
\label{eq:asa}
\pt{\Stu_{\phi}E_{p,\tau}^{\phi}}\pt{b,\nu\pt{p}} & = \dnbt \pt{b},
\end{align}
where $\nu\pt{p}$ is the center of the $p$-frequency band where the basis $D_{p,\tau}$ in \eqref{dostbasis} is supported.
In Section \ref{sec:baseephi}, we introduce a fast -- $\mathcal{O}\pt{N \log N}$ -- algorithm to compute the coefficients
\begin{align*}
\pt{f, \ephipt}_{\Lloi}.
\end{align*} 
Unfortunately, for a general admissible  window $\phi$, the basis \eqref{eq:baseephip} fails to be orthogonal. 
Nevertheless, under a mild condition on $\phi$, we prove that it forms a frame, which in general is not tight.
So, by abstract theory of frames, we obtain that the coefficients in \eqref{eq:framedual} are
\begin{align*}
\pt{f, \widetilde{\ephipt}}_{\Lloi},
\end{align*}
where $\widetilde{\ephipt}$ is the dual frame of $\ephipt$.

%
%

This paper is organized as follows: in Section \ref{sec:intro}, we provide a brief survey on the $\ststst$-transform  
in the context of  time-frequency analysis. In particular, we point out the similarities and the differences between Fourier transform, short-time Fourier transform and wavelet transform. 
In Section \ref{sec:propbase}, we prove that \eqref{dostbasis} is a basis of $\Lloi$ and we highlight its time-frequency local properties.
In Section \ref{sec:basestrock}, we decompose the Stockwell transform with a general window using \eqref{dostbasis}. Moreover,
we determine the explicit expression of $(\ststst_\phi \dnbt)$. In Section \ref{sec:discr}, we provide a 
discretization of the $\ststst$-transform.
In Section \ref{sec:baseephi},  we  determine the basis \eqref{eq:baseephip} adapted to a general admissible window $\phi$. 
We propose an algorithm which evaluates the coefficients related to the basis \eqref{eq:baseephip} 
of computational complexity $\mathcal{O}(N \log N)$. This algorithm extends the one proposed
by Y. Wang and J. Orchard in \cite{WO09}.


\section{A Brief Survey on the \texorpdfstring{$\ststst$-transform}{S-transform}}
\label{sec:intro}
In many practical applications it is important to analyze signals, $i.e.$ extracting the
time-frequency content of a signal. 
Given a signal $f$ in $\Ll$, we can precisely extract its frequency content using the Fourier transform $\F$
\begin{align*}
  \widehat{f}\pt{\xi} &=\pt{\F f} (\xi) =\pt{2\pi}^{-1/2}\int_{\Rr}e^{-\pii t\xi}\,f(t)\, dt,\qquad \xi\in\Rr.
\end{align*}
Unfortunately, due to uncertainty principle, it is impossible to retain at the same time precise time-frequency information.
In the past years, many techniques arose trying to deal with the uncertainty principle in order to obtain a \emph{sufficiently good} time-frequency
representation of a signal. 
The short-time Fourier transform 
\begin{align*}
  \Gtp{f} \bxi & =\pt{2\pi}^{-1/2}\int_{\Rr} e^{-\pii t\xi}\, f\pt{t}\overline{\phi\pt{t-b}}\, dt,\qquad b,\xi\in\Rr
\end{align*}
is one of the standard tools.
Loosely speaking, taking the short-time Fourier transform of a signal $f$ at a certain time $b$ is like taking the Fourier
transform of the signal $f$ cut by a window function $\phi$ centered in $b$, see for example 
\cite{FEGR97}, \cite{GR01}.  
It is possible to invert the short-time Fourier transform using the following theorem.

\begin{thm}
\label{STFTInversion}Let $f$ be a signal in $\Ll$ and $\phi$ a window in $L^2\pt{\R}$. Then 
\begin{align*}
\widehat{ f}\pt{\xi}= & \int_{\Rr}\Gtp{f}\pt{b,\xi}\, db,\qquad \xi\in\Rr.
\end{align*}
\end{thm}
Notice that the width of the analyzing window remains fixed. Due to the Nyquist sampling theorem, 
it would be natural to consider a  window whose width depends on the analyzed frequency. 
To accomplish this task, in  \cite{stockwell1996localization}, the $\ststst$-transform $\ststst_g$ was introduced as
\begin{align}
\pt{\Stu_g {f}}\pt{b,\xi} & =\pt{2\pi}^{-1/2}\abs{\xi}\int_{\Rr}e^{-\pii t\xi}\,f\pt{t}e^{-\pt{t-b}^{2}\xi^{2}/2}\, dt,\qquad b,\xi\in\Rr.\label{eq:cont_stoc_time-1}
\end{align}
Notice that the width of the Gaussian window $e^{-\pt{t-b}^{2}\xi^{2}/2}$ shrinks as the analyzed frequency increases, providing a better time-localization for high frequencies. 
It is possible to rewrite the Stockwell transform with respect to
the Fourier transform of the signal $f$ as
\begin{align}
\pt{\Stu_g {f}}\pt{b,\xi} & =\int_{\Rr} e^{\pii \zeta b}\,\widehat{ f}\pt{\zeta +\xi } e^{-\frac{2\pi^{2}\zeta^{2}}{\xi^{2}}}\,
d\zeta,\qquad b,\xi \in\Rr, \; \xi\neq 0.\label{eq:cont_stoc}
\end{align}
In \cite{stockwell1996localization} it has been stated an inversion formula similar to Theorem \ref{STFTInversion}. 

\begin{thm}
\label{StoGausInv}Let $f$ be a signal in $\Ll.$ Then
\begin{align*}
\widehat{ f}\pt{\xi} & =\int_{\Rr}\pt{\Stu_g{f}}\pt{b,\xi}\, db,\qquad \xi\in\Rr.
\end{align*}
\end{thm}

Many extensions of this transform have been suggested in the last years. 
See for example \cite{djurovic2008frequency,GMS09,cai2011seismic,WZ07,YZ11}. We here recall the one introduced in \cite{WZ07}. 

\begin{Def}
\label{StoWong}Let $f$ be a signal in $\Ll$ and let $\phi$ be a window function in $\Ll$. 
Then, we call %
\begin{align}
\label{eq:StoWong}
\pt{\Stu_{\phi}{f}}\pt{b,\xi} & =\pt{2\pi}^{-1/2}\int_{\Rr} e^{-\pii t\xi}\, f\pt{t}\abs{\xi}\overline{\phi\pt{\xi\pt{t-b}}}\, dt,\qquad b,\xi\in\Rr
\end{align}
the $\ststst_\phi$-transform of the signal $f$ with respect to the window $\phi$. 
\end{Def}

\noindent It is possible to recover the original definition (\ref{eq:cont_stoc_time-1}) taking $\phi$ 
to be the Gaussian window $\phi\pt{t}=e^{-t^{2}/2}.$ The $\ststst$-transform has been recently extended 
to the multi-dimensional case by the second author \cite{RW14}. 
Theorem \ref{StoGausInv} still holds for the $\ststst$-transform \eqref{eq:StoWong}.

See \cite{BIDA13,CH14,DSR09,GZBM04,BH13,LA14,GA13,ZU13} for some applications of the $\ststst$-transform to signal processing.

Heuristically, we can think at the $\ststst$-transform as a short-time Fourier transform in which the width 
of the analyzing window varies with respect to the analyzed frequency. 
Therefore, the $\ststst$-transform can also be interpreted as a particular 
non stationary Gabor transform, see \cite{BDJH11}.

We can give an equivalent definition of the $\ststst$-transform using the following proposition.

\begin{pipi}\label{prop:fourier}
Let $f  $ be a signal in $\Ll$ and let $\phi$ be a window in $\Ll$. Then 
\begin{align*}
\pt{\Stu_{\phi}{f}}\pt{b,\xi} & = e^{-\pii b\xi }\pt{\Fi_{\zeta\mapsto b}{f_\xi}}\pt{b},\qquad b,\xi\in\Rr,\quad \xi\neq 0,
\end{align*}
where
\begin{align*}
f_\xi \pt{\zeta} & = \widehat{f}\pt{\zeta} \,\overline{\widehat{\phi}\pt{\frac{\zeta -\xi }{\xi}}},\qquad \zeta\in\Rr,\;\xi\neq 0.
\end{align*}
\end{pipi}
The following inversion formula has been proven in \cite{WZ07}. 
\begin{thm} 
 \label{ResForStock}
Let $\phi$ be a function in $\L\cap\Ll$ such that 
\begin{align*}
c_\phi & = \int_\Rr \abs{ \widehat{\phi} \pt{\xi} }^2 \,\frac{d\xi}{ \abs{\xi+1} } <\infty .
\end{align*}
We say that $\phi$ is an admissible window for the $\ststst$-transform and we call $c_\phi$ the admissibility constant.
Then 
\begin{align*}
c_{\phi}\pt{f,f'}_{\Ll}= & \int_{\Rr}\int_{\Rr}\pt{\Stu_{\phi}{f}}\pt{b,\xi}\overline{\pt{\Stu_{\phi}{f'}}\pt{b,\xi}}\, db\frac{d\xi}{\abs{\xi}},
\end{align*}
for all $f$ and $f'$ in $\Ll$.
\end{thm}

At this point, it is useful to recall the wavelet transform $\Wa{\phi}{f}$
of a signal $f$ in $\Ll$ with respect to the window $\phi$ 
\begin{align*}
\Wap{\phi}{f}\ba & =\int_{\Rr}f\pt{t}\abs{a}^{-1/2}\overline{\phi\pt{a^{-1}\pt{t-b}}}\, dt,\qquad\forall b,a\in\Rr.
\end{align*}
See for example \cite{BE98,DA92,MA09} for details on wavelet analysis and filter banks.
\begin{thm}
\label{ResForWav}Let $\phi$ be a window in $\Ll$ such that 
\begin{align*}
c_{\phi} & =\int_{\Rr}\abs{\widehat{\phi}\pt{\xi}}^{2}\frac{d\xi}{\abs{\xi}}<\infty.
\end{align*}
We say that $\phi$ is an admissible wavelet and we call $c_\phi$ the admissibility constant.
Then 
\begin{align*}
c_{\phi}\pt{f,f'}_{\Ll} & =\int_{\Rr}\int_{\Rr}\Wap{\phi}{f}\ba\overline{\Wap{\phi}{f'}\ba}\, db\frac{da}{a^{2}},
\end{align*}
for all $f$ and $f'$ in $\Ll$.
\end{thm}
\noindent Notice the similarities between Theorem \ref{ResForStock} and Theorem \ref{ResForWav}. This follows from a deep connection
among Stockwell transform, short-time Fourier transform and wavelet transform. In fact, these transforms 
are related to the affine Weyl-Heisenberg group studied in \cite{kalisa1993n}. This connection has been highlighted in the multi-dimensional case by the second author in \cite{RI14}.  
In \cite{GLM06,VE08}, the connections between Stockwell transform and wavelet transform are pointed out.
The affine Weyl-Heisenberg group is also connected to the definition of $\alpha$-modulation spaces, see
 \cite{BA14,DA08,FO07}, which represents, at the level of coorbit theory, a sort of interpolation between
 Modulation spaces and Besov spaces. 
 A different group approach to the Stockwell transform has been studied in \cite{boggiatto2009group}.

\section{A Time-Frequency Localized Basis}
\label{sec:propbase}

In this section, we prove that the system of functions \eqref{dostbasis}, proposed by R. G. Stockwell  in \cite{ST07}, 
is indeed an orthonormal basis of $\Lloi$.

\begin{figure}[t]
\centering
\includegraphics[width=.8\textwidth]{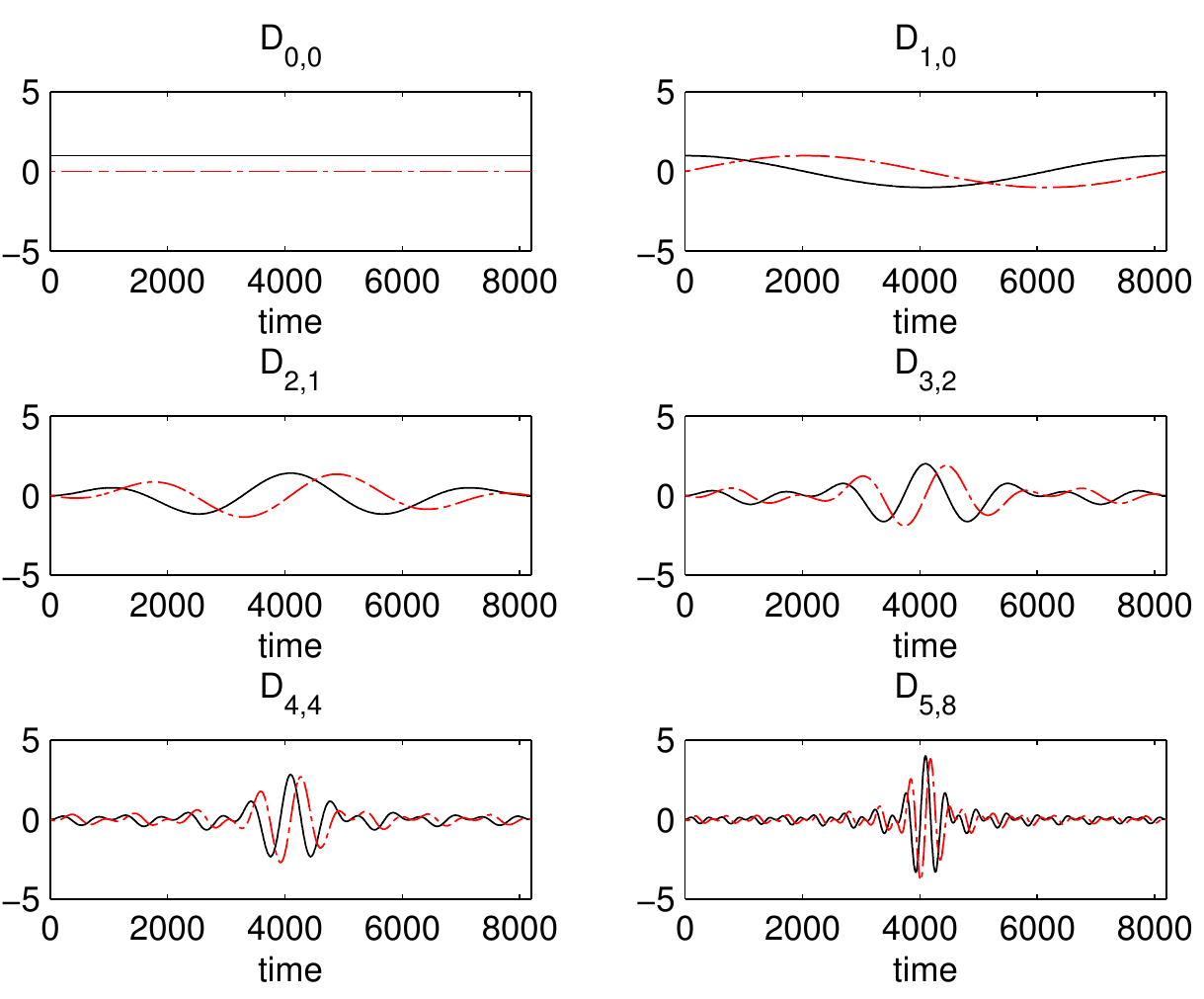}
\caption{DOST basis functions in increasing frequency $p$-bands. Black line = real, red line = imaginary. 
See Figure 2 in \cite{ST07} for a comparison.}
\label{fig:01}
\end{figure}

For $p=0$, we define
\begin{align*}
 \nu(0)=0,\quad \beta(0)=1, \quad \tau(0)=0,
\end{align*}
for $p=1$
\begin{align*}
 \nu(1)=1,\quad \beta(1)=1, \quad \tau(1)=0,
\end{align*}
for all $p \geq 2$
\begin{align*}
 \nu(p)= 2^{p-1}+ 2^{p-2}, \quad \beta(p)= 2^{p-1},\quad \tau(p)= 0, \ldots, \beta(p)-1.
\end{align*}
Setting, for each $p$, the $p$-frequency band  
\begin{align*}
\ptq{\beta(p), 2 \beta(p)-1}= \left[\nu(p)- \frac{\beta(p)}{2},\nu(p)+ \frac{\beta(p)}{2}-1\right],
\end{align*}
we obtain a partition of $\N$; notice that $\nu(p)$ is the center of each $p$-frequency  band.
We recall here the definition of the so-called DOST functions, introduced in \cite{ST07}:
\begin{align*}
D_0\pt{t}& =1,\qquad t\in\Rr, \\ 
D_1 \pt{t}&= e^{ \pii  t },\qquad t\in\Rr,
\end{align*}
and
\begin{align*}
D_p  & =\ptg{\dnbt (t)}_{\tau=0, \ldots, \beta(p)-1},\qquad t\in\Rr,
\end{align*}
where
\begin{align*}
\dnbt (t) & =\frac{1}{\sqrt{\beta(p)}}\sum_{f= \nu(p)-\beta(p)/2}^{\nu(p)+\beta(p)/2-1} e^{\pii f t} e^{- \pii f \tau/\beta(p)},\qquad t\in\Rr.
\end{align*}
For all negative integers $p$, we set 
\begin{align*}
\dnbt (t)= \overline{\dnbtneg(t)},\qquad \tau=0,\ldots, \beta\pt{\abs{p}}-1.
\end{align*}
For each $p \in \N$, $\nu(-p)= - \nu(p)$ and $\beta(-p)=\beta(p)$.
In the sequel we  call 
\begin{align}
\label{eq:newbasis}
 \bigcup_{p \in \Z} D_p
\end{align}
Stockwell basis.

Notice that, in the original paper \cite{ST07}, each $\dnbt $ had a multiplicative factor $e^{ \tau \pi \mathrm{i} }$. 
Since this factor is not crucial in  proving that \eqref{eq:newbasis} is a basis of $\Lloi$, we have decided to
drop it. In \eqref{eq:stcoeff}, we clarify the role of this multiplicative factor. 
In Figure \ref{fig:01} and  Figure \ref{fig:02} we draw the DOST basis functions 
without this multiplicative factor. In Figure \ref{fig:02} notice that, with our choice, 
these functions are self-similar in each $p$-band, in contrast to the ones defined in \cite{ST07}.
Moreover, we have slightly changed the notation in the frequency domain.
The $k$th element of the Fourier basis is $e^{\pii  k t }$, 
while, in the original paper, the $k$th element is $e^{-\pii k t}$.
The convention we adopt seems closer to the standard Fourier analysis.

\begin{figure}[t]
\centering
\includegraphics[width=.8\textwidth]{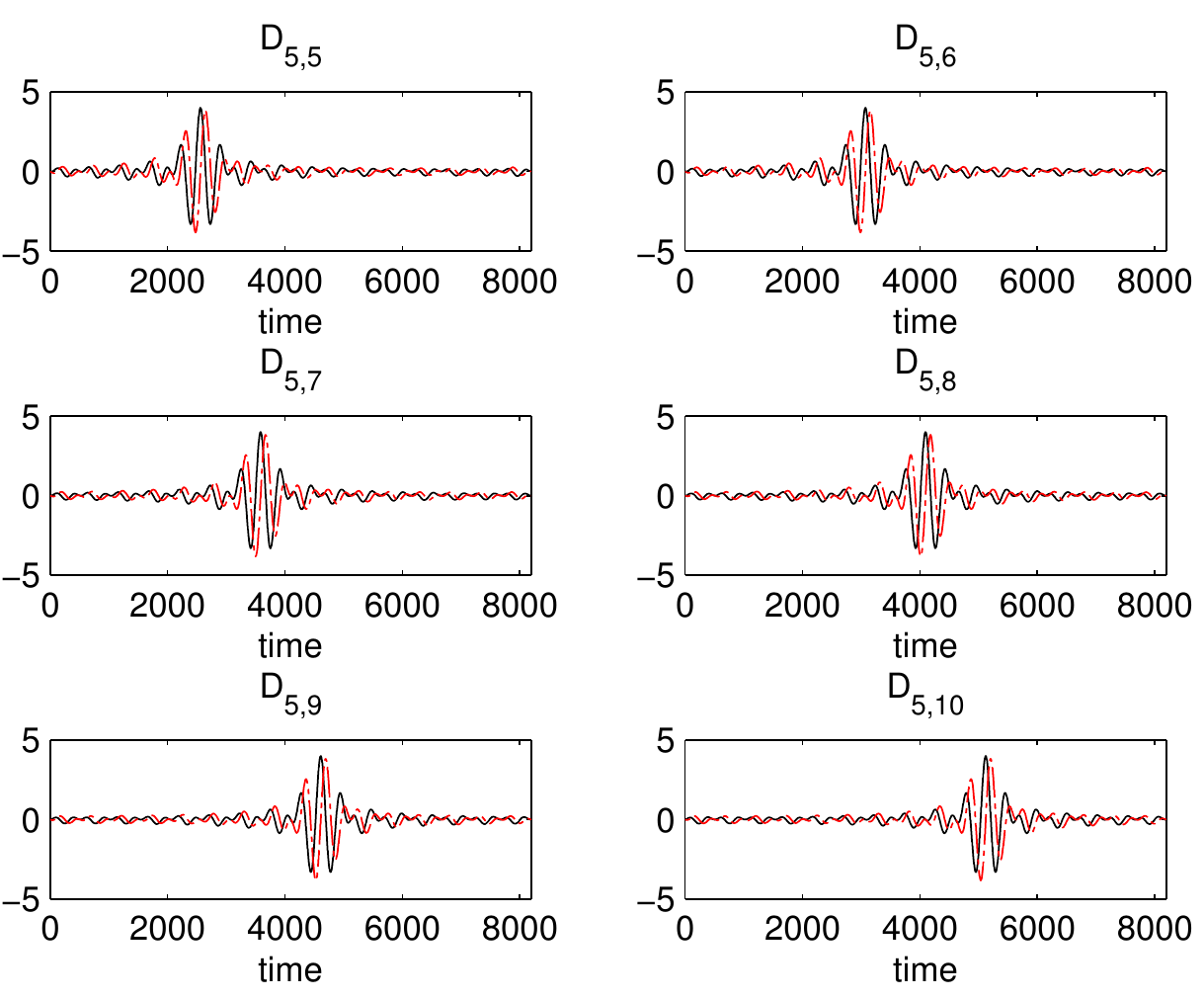}
\caption{DOST basis functions in the same $p$-band ($p=5$). Black line = real, red line = imaginary. 
See Figure 1 in \cite{ST07} for a comparison. }\label{fig:02}
\end{figure}

\begin{thm}
\label{thm:basidost}
$\,\bigcup_{p \in \Z} D_p$ is an orthonormal basis of $\Lloi$.
\end{thm}

\begin{proof}
In the sequel, we  consider positive $p$. For negative $p$, all results hold true using the adjoint property.
We recall that $\{e^{\pii k t}\}_{k \in \Z}$ is an orthonormal basis of $\Lloi$ and we notice that $\dnbt(t)$
is a finite linear combination of $e^{\pii k t}$ with $k$ in the $p$-frequency band 
\begin{align*}
I_p =\ptq{\nu(p)-\frac{\beta(p)}{2},\nu(p)+\frac{\beta(p)}{2}-1}.
\end{align*}
Hence, we can conclude that 
\begin{align*}
\pt{\dnbt,D_{p', \tau'}}_{\Lloi} =0, \quad \mbox{if }p\neq p',\quad \forall \tau, \tau',
\end{align*}
since the $p$-band and the $p'$-band are disjoint. So, we can focus on the case $p=p'$.
The proof is divided into three steps.

\medskip
\begin{description}
\item[Step I - ] $\norm{\dnbt }_{\Lloi}=1$.
\end{description}
Consider the inner product 
\begin{multline*}
\norm{\dnbt }^2_{\Lloi} = \pt{\dnbt,\dnbt}_{\Lloi}\\ 
=\frac{1}{\beta(p)}\int_{0}^1 
\left( \sum_{f= \nu(p)-\beta(p)/2}^{\nu(p)+\beta(p)/2-1} e^{\pii f t} e^{-\pii f \tau/\beta(p)}\right) 
\left( \sum_{f'= \nu(p)-\beta(p)/2}^{\nu(p)+\beta(p)/2-1} e^{-\pii f' t} e^{\pii  f' \tau/\beta(p)}\right)dt.
\end{multline*}

Since $\{e^{\pii k t}\}_{k \in \Z}$ is an orthonormal basis, 
\begin{align*}
 \norm{ \dnbt}_{\Lloi}^2= \frac{1}{\beta(p)}\sum_{f= \nu(p)-\beta(p)/2}^{f= \nu(p)+\beta(p)/2-1}\int_0^1 1\; dt= 1. 
\end{align*}

\medskip
\begin{description}
\item[Step II - ]  $\bigcup_{p \in \Z} D_p$ is an orthonormal set.
\end{description}

If $p\neq p'$ the $L^2$-scalar product vanishes, so we can suppose $p=p'$.
It is convenient to consider $j= f -\beta(p)$.
\begin{align}
\nonumber \dnbt (t) =& \frac{1}{\sqrt{\beta(p)}}\sum_{j=0}^{\beta(p)-1} e^{ \pii (\beta(p)+j)t} e^{-\pii (\beta(p)+j) \tau/\beta(b)}\\
 \label{eq:Dpriscritta}=&\frac{1}{\sqrt{\beta(p)}}\sum_{j=0}^{\beta(p)-1} e^{ \pii (\beta(p)+j)t}  e^{-\pii \tau j/\beta(p)}.
\end{align}
The orthonormality of the Fourier basis implies
\begin{equation}
 \label{eq:proddn}
 \pt{\dnbt, D_{p, \tau'} }_{\Lloi}= \frac{1}{\beta(p)} \sum_{j=0}^{\beta(p)-1} e^{\pii (\tau'- \tau) j/\beta(p)}.
 \end{equation}
 Now, we need the following lemma.
 \begin{lem}
 \label{lem:rad1}
  Let $k \in \N\setminus \{0\}$. Then
  \begin{align}
  \label{eq:somrad}
   \sum_{j=0}^{2^k-1} e^{\pii j m /2^k}=0, \quad m=\pm 1, \ldots, \pm (2^k-1).
  \end{align}
 \end{lem}
 \begin{proof}
  Notice that \eqref{eq:somrad} is a truncated geometric series with ratio
  $e^{\pii m/2^k}$. Therefore,  the well known formula for geometric progression implies that
  \[
    \sum_{j=0}^{2^k-1} e^{\pii j m /2^k}= \frac{1- e^{\pii m 2^k/2^k}}{1- e^{\pii m/2^k}}= 0.
  \]
  Since $m=\pm1, \ldots, \pm(2^{k}-1)$, the denominator in the above equation never vanishes.
 \end{proof}

Let $(\tau'- \tau)= m$ in \eqref{eq:somrad}, then Lemma \ref{lem:rad1} implies that
\begin{align*}
 \pt{\dnbt,D_{p', \tau'} }_{\Lloi}=\delta_{0}(p-p')\delta_{0}(\tau-\tau'),
\end{align*}
i.e. $\bigcup_{p \in \Z} D_p$ is an orthonormal set.

\medskip
\begin{description}
\item [Step III - ] $\bigcup_{p \in \Z}D_p$ is a basis of $\Lloi$.
\end{description}
Notice that
\begin{align*}
 D_p \subseteq \Span\{ e^ { \pii k t} \}_{k \in [ \beta(p), 2\beta(p)-1]}.
\end{align*}
Hence, to prove the assertion it is sufficient to show that the elements of the set $\ptg{\dnbt}_{\tau=0, \ldots, \beta(p)-1} $
are a basis of $ \Span \{ e^ { \pii k t} \}_{k \in [ \beta(p), 2\beta(p)-1]}$. Since we deal with finite dimensional vector spaces,
we prove that  the functions $\ptg{\dnbt}_{\tau=0}^{\beta(p)-1}$ are linearly independent; that is
\begin{align}
 \label{eq:a}
  \sum_{\tau=0}^{\beta(p)-1} c_\tau \dnbt =0 \Longrightarrow c_\tau=0, \quad \forall \tau=0, \ldots, \beta(p)-1.
\end{align}
Since $\{ e^ { \pii (\beta(p)+j) t} \}_{j=0, \ldots,  \beta(p)-1}$ is a basis, we can consider the projection of \eqref{eq:a} on each term
$\{e^{\pii (\beta(p)+j) t}\}_{j=0, \ldots, \beta(p)-1}$ of
the Fourier basis. We obtain the system
\begin{align}
\label{eq:base}
 \sum_{\tau=0}^{\beta(p)-1} c_\tau e^{-\pii \tau j/ \beta(p)}=0, \quad j=0, \ldots, \beta(p)-1.
\end{align}
Notice that \eqref{eq:base} can be written as the linear system
\begin{align}
 \label{eq:vander}
  \left(\begin{array}{cccc}
    1 &  1  &\ldots  &1 \\
    1 & e^{- \pii/\beta(p)} & \ldots& e^{-\pii \pt{\beta(p)-1}/\beta(p)}\\
    1 & e^{- \pii 2 /\beta(p)} & \ldots& e^{-\pii 2 \pt{\beta(p)-1}/\beta(p)}\\
    \vdots &\vdots & \ddots&\vdots \\
    1 & e^{-\pii \pt{\beta(p)-1}/\beta(p) }&\ldots& e^{-\pii \pt{\beta(p)-1}^2/\beta(p) }
  \end{array}\right)\cdot
  \left(\begin{array}{c}
   c_0\\
   c_1\\
   c_2\\
   \vdots\\
   c_{\beta(p)-1}
  \end{array}\right)
   =
   \left(\begin{array}{c}
   0\\
   0\\
   0\\
   \vdots\\
   0
  \end{array}\right)
\end{align}
The square matrix in \eqref{eq:vander} is a Vandermonde matrix 
with entries $\ptg{e^{-\pii l/\beta(p)}}_{l=0}^{\beta(p)-1}$.
Since the entries are all distinct the determinant of the Vandermonde matrix is non 
zero and the unique solution of the linear system \eqref{eq:vander} is the zero vector. That
is the functions $\ptg{\dnbt}_{\tau=0}^{\beta(p)-1}$ are linear independent.
\end{proof}

Lemma \ref{lem:rad1} implies the following corollary.
\begin{cor}
\label{pro:nul}
 For each $p\in \Z$ and each $\tau, \tau'= 0, \ldots, \beta(|p|)-1$ we have
 \begin{align*}
  \dnbt \pt{\frac{\tau'}{\beta(p)}}= \sqrt{\beta(p)} \delta_0(\tau'-\tau).
 \end{align*}
\end{cor}

\begin{proof}
 Let us suppose $p$ positive. Then 
 \begin{align*}
    \dnbt \left(\frac{\tau'}{\beta(p)}\right) &= \frac{1}{\sqrt{\beta(p)}}\sum_{j=0}^{\beta(p)-1} 
    e^{\pii (\beta(p)+j) \left(\frac{\tau'-\tau}{\beta(p)}\right)}\\
    & =\frac{1}{\sqrt{\beta(p)}}\sum_{j=0}^{\beta(p)-1} 
    e^{\pii j (\tau'-\tau)/\beta(p)}.
  \end{align*}
Letting $(\tau'-\tau)=m$, we apply Lemma \ref{lem:rad1} and we obtain the assertion. For negative $p$ we use the adjoint property.
\end{proof}

The DOST functions are not dilations nor translations of a single function. 
Nevertheless, for each $p$,
\begin{align*}
 D_p\pt{t}= \left\{ \frac{1}{\sqrt{\beta(p)}} \sum_{j=0}^{\beta(p)-1} e^{\pii (\beta(p)+j) (t-\tau/\beta(p))}\right\}_{\tau=0, \ldots, \beta(p)-1}
\end{align*}
is formed by translations of $\tau/\beta(p)$ of the same function. Roughly speaking, we can state that the DOST basis is not self similar globally, but it is self
similar in each band, see Figure \ref{fig:02}. 
Hence, the $\ststst$-transform in this setting appears different from the wavelet transform because the mother wavelet changes 
as the frequencies increases, in contrast to the usual formulation.


R. G. Stockwell proposed this basis because it is an efficient compromise between frequency localization in low 
frequencies and time localization for high frequencies.
The price to pay is that, on one hand, for high frequencies, we do not have a precise frequency localization, 
but just a localization in a certain band, which is wider as the frequency increases and, on the other hand, in low frequencies, 
we lose
time localization. 
In fact, for high frequencies, the basis $\dnbt$ are,
in large sense, local at $t=\tau/\beta(p)$.
It is not true that $\dnbt$ has compact support in time,
but the energy is concentrated near the point
$t=\tau/\beta(p)$. We prove that basis $\dnbt $ are $0.85$-concentrated in the neighborhood
\begin{align*}
I_{p, \tau}=
\ptq{\frac{\tau}{\beta(p)}-\frac{1}{2 \beta(p)}, \frac{\tau}{\beta(p)}+\frac{1}{2 \beta(p)}},
\end{align*}
in the sense of the Donoho-Stark Theorem \cite{DS89, CA14}.  
\begin{pipi}
\label{prop:loc}
 For each $\dnbt(t)$ we have 
 \begin{align*}
 \norm{\dnbt}_{L^2\pt{I_{p,\tau }}}=
 \left(\int_{\frac{2\tau-1}{2 \beta(p)}} ^{\frac{2\tau+1}{2 \beta(p)}} 
 |\dnbt|^2 dt\right)^{1/2}> 0,85,
 \end{align*}
$i.e.$ the $L^2$-norm is concentrated in the interval 
\begin{align*}
I_{p, \tau} &= \left[\frac{\tau}{\beta(p)}-\frac{1}{2 \beta(p)}, \frac{\tau}{\beta(p)}+\frac{1}{2 \beta(p)}\right].
\end{align*}
Since $\norm{\dnbt}=1$, we can also state that the $L^2$-norm of $\dnbt$ is less that $0,15$ out 
of $I_{p, \tau}$.
For $\tau=0$, $I_{p, 0}$ must be considered as an interval in circle, that is 
\begin{align*}
 I_{p,0}= 
 \left[ 0, \frac{1}{2 \beta(p)}\right) \cup \left(1-\frac{ 1}{2\beta(p)}, 1 \right].
\end{align*}
\end{pipi}
\begin{proof}
 Since in each $p$-band the basis functions are a translation of $\tau/\beta(p)$ of the same function, 
 we can prove the property for a fixed $\tau$. For simplicity,
 we consider  $\tau=0$. In order to take in account just one integral, we extend by periodicity the function for negative $t$ and we evaluate
 \begin{align}
 \label{eq:lolbas}
  \int_{-\frac{1}{2 \beta(p)}}^{\frac{1}{2\beta(p)}} \abs{\dnbt\pt{t}}^2 \,dt.
 \end{align}
Notice that 
\begin{align}
 \label{eq:indutpro}
\nonumber \abs{\dnbto}^2=& \dnbto (t) \cdot \overline{\dnbto (t)}\\
\nonumber =&\frac{1}{\beta(p)}\pt{\sum_{j=0}^{\beta(p)-1} e^{ \pii (\beta(p)+j) t}}\cdot \pt{\sum_{k=0}^{\beta(p)-1} e^{-\pii (\beta(p)+k)t}}\\
=&\frac{1}{\beta(p)}\sum_{m=-\beta(p)+1}^{\beta(p)-1}(\beta(p)-|m|) e^{\pii m t}.
 \end{align}
Equation \eqref{eq:indutpro} can be proven by induction on the size of the band. Writing \eqref{eq:indutpro}
 in terms of cosine and sine we obtain
 \begin{align*}
  &\abs{\dnbto}^2\\
	= &\frac{1}{\beta(p)} \sum_{m=-\beta(p)+1}^{\beta(p)-1}(\beta(p)-|m|) 
  (\cos( 2 \pi m t) + \mathrm{i} \sin(2 \pi m t))\\
  = & 1+\frac{1}{\beta(p)} \sum_{m=1}^{\beta(p)-1}\pt{\beta(p)-m}
  \pt{ \pt{\cos(2 \pi m t)+ \cos(-2 \pi m t)} +\mathrm{i} \pt{\sin ( 2\pi mt)+ \sin (-2 \pi m t)} } \\
  = & 1+ \frac{2}{\beta(p)} \sum_{m=1}^{\beta(p)-1}(\beta(p)-m) \cos\pt{2 \pi m t}.
 \end{align*}
Therefore,
  \begin{align*}
    &\int_{-\frac{1}{2 \beta(p)}}^{\frac{1}{2\beta(p)}} \abs{\dnbto \pt{t}}^2 \,dt\\
    = &\int_{-\frac{1}{2 \beta(p)}}^{\frac{1}{2\beta(p)}} dt + \frac{2}{ \beta(p)}
    \int_{-\frac{1}{2 \beta(p)}}^{\frac{1}{2\beta(p)}} \sum_{m=1}^{\beta(p)-1}(\beta(p)-m) \cos(2 \pi m t)\, dt\\
    = & \frac{1}{\beta(p)}+ \frac{2}{\beta(p)} \sum_{m=1}^{\beta(p)-1} (\beta(p)-m) \frac{\sin(2 \pi m t)}{2 \pi m }\left|^{\frac{1}{2 \beta(p)}}_{- \frac{1}{2 \beta(p)}}\right.\\
    = & \frac{1}{\beta(p)}+ \frac{4}{\beta(p)} \sum_{m=1}^{\beta(p)-1} (\beta(p)-m) \frac{\sin\left( \frac{2 \pi m}{2\beta(p)} \right)}{2 \pi m }.
  \end{align*}
By the Maclaurin expansion of $\sin(x)$, 
  \begin{align*}
    \int_{-\frac{1}{2 \beta(p)}}^{\frac{1}{2\beta(p)}} \abs{\dnbto}^2 dt= & \frac{1}{\beta(p)}+ \frac{4}{\beta(p)} \sum_{m=1}^{\beta(p)-1} (\beta(p)-m)\left( \frac{2\pi m \frac{1}{2\beta(p)} +R_m(\eta)}{2 \pi m } \right),
  \end{align*}
where $R_m(\eta)$ is the Lagrange rest. Using Leibniz summation formula we obtain
 \begin{align*}
    \int_{-\frac{1}{2 \beta(p)}}^{\frac{1}{2\beta(p)}} \abs{\dnbto\pt{t}}^2 dt    \cong & \frac{1}{\beta(p)}+ \frac{2}{\beta^2(p)} \sum_{m=1}^{\beta(p)-1} (\beta(p)-m)\\
    = & \frac{1}{\beta(p)}+ \frac{2}{\beta^2(p)} \left(\beta(p) (\beta(p)-1)- \frac{1}{2} \beta(p) (\beta(p)-1) \right)\\
    = & \frac{1}{\beta(p)}+ \frac{1}{\beta(p)}  (\beta(p)-1)=1.\\
 \end{align*}     
We have to take in account the rests $R_m(\eta)$. Since 
\begin{align*}
\sup \abs{\frac{\mathrm{d}^3}{\mathrm{dt}^3}\ptq{\sin(2 \pi m t)}}= (2 \pi m)^3, 
\end{align*}
we can conclude that
\begin{align*}
 &\left|\frac{4}{\beta(p)} \sum_{m=1}^{\beta(p)-1}(\beta(p)-m) \frac{R_m(\eta)}{2 \pi m}\right|\\
 \leq & \frac{4}{\beta(p)} \sum_{m=1}^{\beta(p)-1} \frac{(\beta(p)-m)}{2 \pi m} \frac{(2 \pi m)^3}{6 (2 \beta(p))^3}\\
 \leq & \frac{\pi^2}{3 \beta^4(p)} \sum_{m=1}^{\beta(p)-1} (\beta(p)-m) m^2 \\
 \leq & \frac{\pi^2}{3 \beta^4(p)}\left(\frac{\beta(p)^2}{6} (\beta(p)-1) (2 \beta(p)-1)- 
 \frac{\beta(p)^2}{4}(\beta(p)-1)^2 \right)\\
 \leq &\frac{\pi^2 (\beta(p)-1)}{3 \beta^2(p)}\left(\frac{\beta(p)}{3} -\frac{1}{6}- \frac{\beta(p)}{4}+ \frac{1}{4} \right)\\
 \leq &\frac{\pi^2 }{36} \frac{\beta(p)^2-1}{\beta(p)^2}\\
 \leq &\frac{\pi^2}{36} < 0,275.
\end{align*}
Hence, finally
\begin{align*}
 \left(\int_{-\frac{1}{2 \beta(p)}}^{\frac{1}{2\beta(p)}} \abs{\dnbto\pt{t}}^2 \,dt\right)^{1/2}\geq (1-0, 275)^{1/2}\geq 
 \sqrt{0,725}> 0,85.
\end{align*}

\end{proof}

\section{Diagonalization of the \texorpdfstring{$\ststst$-transform}{S-transform} }
\label{sec:basestrock}
In this section, for the sake of clarity, we write $\ststst_\phi$-transform instead of $\ststst$-transform to emphasize the window dependence.  We focus our attention to $L^2\pt{[0,1]}$. 
Using Fourier series, it is well known that if $f \in L^2\pt{[0,1]}$, then 
\begin{align*}
 f(t)&= \sum_{k \in \Z} \hat{f}(k) e^{\pii k t}, \quad a.e., 
\end{align*}
and
\begin{align*}
 \norm{f}_{L^2\pt{[0,1]}}&= \pt{\sum_{k\in \Z} |\hat{f}(k)|^2 }^\frac{1}{2}.
\end{align*}

We define the Hilbert space 
$\pt{Y, \pt{\phantom{f} ,\phantom{f}  }_Y, \norm{\phantom{f}}_Y}$:
\begin{align*}
  Y&= \ptg{\sum_{k \in \Z} c_k(\xi) e^{\pii (k-\xi)b }\mid c_k(\xi) \in L^2\pt{\Rr, \frac{1}{|\xi|}}, \; \mbox{and }\;
  \sum_{k \in \Z} \norm{c_k}^2_{L^2\pt{\Rr, \frac{1}{|\xi|}}} <\infty},\\
  \quad (g, g')_Y&= \int_{0}^1 \int_{\Rr} g(b, \xi) \overline{g'(b, \xi)}  \frac{d\xi}{|\xi|} d b, \quad g, g' \in Y,\\
  \norm{g}_Y&=\sqrt{\pt{g,g}_Y} = \pt{\sum_{k \in \Z} \norm{g_k}_
  {L^2 \pt{\Rr, \frac{1}{|\xi|} }}^2  }^{\frac{1}{2}}, \quad g(b, \xi)= \sum_{k \in \Z} g_k(\xi) e^{\pii (k-\xi) b}, \quad a.e.\,.
\end{align*}
 In view of Theorem  \ref{ResForStock}, we introduce $\pt{Z, \pt{\phantom{f} ,\phantom{f}  }_Z, 
 \norm{\phantom{f}}_Z}$
 the Hilbert space of admissible windows:
 \begin{align}
 \label{eq:defZ}
   Z&=\ptg{\phi \in \Si'(\R) \mid \int |\widehat{\phi}\pt{\xi}|^2\frac{d\xi}{\abs{1+\xi}}<\infty},\\
  \nonumber
  \quad \pt{\phi, \phi'}_{Z}&=\int \widehat{\phi}\pt{\xi} \overline{\widehat{\phi'}}\pt{\xi} 
   \frac{d\xi}{\abs{1+\xi}}, \quad \phi, \phi'\in Z,\\
  \nonumber
   \norm{\phi}_Z&= \sqrt{\pt{\phi, \phi}_Z}= \pt{\int |
   \widehat{\phi}\pt{\xi}|^2\frac{d\xi}{\abs{1+\xi}}}^{1/2}.
 \end{align}

 \begin{thm}
 \label{thm:contl2}
 We define
 \begin{align*}
  \ststst: \Lloi \times \pt{Z\cap \Si\pt{\R}} &\longrightarrow Y\\
   \pt{f, \phi}=\pt{\sum_{k \in \Z} \widehat{f} \pt{k} e^{\pii k t}, \phi}& \longmapsto 
   \sum_{k \in \Z} \widehat{f}\pt{k} \pt{\ststst_\phi\pt{e^{\pii k \cdot} }} \pt{b, \xi}
 \end{align*}
 where, in view of Proposition \ref{prop:fourier}, we set
 \begin{align*}
\pt{  \ststst_\phi e^{\pii k \cdot} }\pt{b, \xi}=& e^{- \pii b \xi} \F^{-1}_{\zeta \mapsto b}\pt{ \overline{\widehat{\phi} \pt{\frac{\zeta - \xi}{\xi}}} \delta_k(\zeta)}(b).
 \end{align*}
 Then $\ststst: \Lloi \times \pt{Z\cap \Si\pt{\R}} \longrightarrow Y$ is continuous.
\end{thm}
\begin{proof}
 We start considering $\ststst(e^{\pii k \cdot},\phi)=\pt{\ststst_\phi e^{\pii k \cdot}}$. 
 By definition, 
 \begin{align*}
\pt{  \ststst_\phi e^{\pii k \cdot} }\pt{b, \xi}=& e^{- \pii b \xi} \F^{-1}_{\zeta \mapsto b}
\pt{ \overline{\widehat{\phi} \pt{\frac{\zeta - \xi}{\xi}}} \delta_k(\zeta)}(b)\\
    =& e^{- \pii b\xi } \F^{-1}_{\zeta \mapsto b}\pt{ \overline{\widehat{\phi} \pt{\frac{k - \xi}{\xi}}} \delta_k(\zeta)}(b)\\
    =&e^{- \pii b \xi} \overline{\widehat{\phi}\pt{\frac{k-\xi}{\xi}}} e^{\pii k b }\\
    =& e^{\pii b \pt{k-\xi}} \overline{\widehat{\phi}\pt{\frac{k-\xi}{\xi}}}.
 \end{align*}
 We observe that
 \begin{align}
 \nonumber
  \norm{\overline{\widehat{\phi}\pt{\frac{k - \cdot}{\cdot}}}}^2_{L^2 \pt{[0,1], 
  \frac{1}{|\xi|} } }=& \int_{\Rr} \abs{\overline{\widehat{\phi}\pt{\frac{k-\xi}{\xi}}}}^2\frac{1}{|\xi|} d\xi\\
  \nonumber
  =&\int_{\Rr} \abs{\overline{\widehat{\phi}\pt{\omega -1}}}^2 \frac{|\omega|}{|k|} \frac{|k|}{|\omega|^2}d\omega\\
  =& \int_{\Rr} \abs{\overline{\widehat{\phi} \pt{w}} }^2\frac{dw}{|w+1|}= \norm{\phi}_Z^2.
 \end{align}
Therefore,
\[
 \norm{\ststst\pt{e^{\pii k \cdot},\phi}}_Y^2=
 \norm{ \ststst_{\phi} e^{\pii k \cdot} }^2_Y=\norm{\phi}_Z^2. 
\]
The functions $\ptg{e^{\pii k t}}_{k \in \Z}$ are orthonormal in $\Lloi$. Notice that
\begin{align}
\nonumber
\pt{\ststst_\phi e^{\pii k \cdot}, \ststst_\phi e^{\pii k' \cdot}}_Y  & =\int_0^1 e^{\pii (k-k')b}db \int_\Rr  \overline{\widehat{\phi} \pt{ \frac{k-\xi}{\xi}} } \widehat{\phi} \pt{ \frac{k'-\xi}{\xi} } \frac{d\xi}{|\xi|}\\
 \label{eq:norm}
 & = \norm{\phi}_Z^2 \delta_0(k-k').
\end{align}
Using the definition of $\ststst$ and equation \eqref{eq:norm}, 
we conclude that if
\begin{align*}
f(t)= \sum_{k \in \Z} \hat{f}(k) e^{\pii k t},\qquad a.e.,
\end{align*}
then
\begin{align*}
  \|\pt{\ststst_\phi f}\pt{\cdot, \cdot} \|^2_{Y}&= \pt{\ststst_\phi f, \ststst_\phi f}_Y\\
  &= \sum_{k \in \Z}\sum_{ k'\in \Z} \pt{ \hat{f}(k)\ststst_\phi e^{\pii k \cdot},\hat{f}(k') \ststst_\phi e^{\pii k' \cdot}}_Y\\
  &=\sum_{k\in \Z} \pt{ \hat{f}(k)\ststst_\phi e^{\pii k \cdot},\hat{f}(k) \ststst_\phi e^{\pii k \cdot}}_Y\\
  &=  \sum_{k\in \Z} |\hat{f}(k)|^2 \norm{ \ststst_\phi e^{\pii k \cdot}}^2_Y\\
  &= \norm{\phi}^2_Z \|f\|^2_{\Lloi}.
\end{align*}
Therefore, $\ststst: \Lloi \times \pt{Z \times \Si\pt{\R}}\to Y$ is a continuous operator.
\end{proof}
\begin{lem}
 Let $\ststst : \Lloi \times \pt{Z \cap \Si\pt{\R}}\to Y$ defined as in Theorem \ref{thm:contl2}.
 Then, since $\Si\pt{\R}\cap Z$ is dense in $Z$, we can extend by continuity $\ststst$ to the whole
 of $Z$.
\end{lem}

\begin{rem}
 Theorem \ref{thm:contl2} is the discrete counterpart of Theorem \ref{ResForWav} in the case of periodic functions. 
\end{rem}

In Section \ref{sec:propbase}, we proved that the DOST functions form an orthonormal 
basis of $\Lloi$. 
Let us assume that $\phi$ belongs to $Z$ defined in \eqref{eq:defZ}.
Then,  by Theorem \ref{thm:contl2}, $\ststst_\phi: \Lloi\to Y$ is continuous.
So, we can write
\begin{align}
 \nonumber \pt{\ststst_\phi f} &= \pt{\ststst_\phi \sum \pt{f, \dnbt}_{\Lloi}\dnbt} \\
 &\nonumber=\sum \pt{f, \dnbt}_{\Lloi} \pt{\ststst_\phi \dnbt} \\
 &\label{eq:dost} =\sum \fdnbt \pt{\ststst_\phi \dnbt},
 \end{align}
where
\begin{align*}
 \fdnbt= \pt{f, \dnbt}_{\Lloi}
\end{align*}
and the sum in \eqref{eq:dost} is over all $\dnbt$ functions.
Hence, in order to understand the $\ststst_\phi$-transform of a general function $f\in\Lloi$, it is sufficient
to evaluate the coefficients $\fdnbt$ and
determine once for all the $\ststst_\phi$-transform of $\dnbt$. 

Notice that, for $p>0$,
\begin{align*}
 \dnbt \pt{t} = &\frac{1}{\sqrt{\beta(p)}}\sum_{j=0}^{\beta(p)-1} e^{ \pii (\beta(p)+ j)(t-\tau/\beta(p))}\\
 = & \frac{1}{\sqrt{\beta(p)}}\sum_{j=0}^{\beta(p)-1} T_{- \tau/\beta(p)} M_{\beta(p)+j} \mathds{1} \pt{t}.
\end{align*}
Hence, we can write
 \begin{align}
\label{eq:foudos}
\nonumber \pt{\F \dnbt}\pt{\xi}  = & \frac{1}{\sqrt{\beta(p)}} \sum_{j=0}^{\beta(p)-1} \pt{\F \T_{- \tau/\beta(p)} \M_{(\beta(p)+j)} \mathds{1} } \pt{\xi}\\
 \nonumber  = &\frac{1}{\sqrt{\beta(p)}} \sum_{j=0}^{\beta(p)-1} \pt{\M_{-\tau/\beta(p)}  \T_{-\beta(p)-j}\F \, \mathds{1} }\pt{\xi}\\
 \nonumber = &\frac{1}{\sqrt{\beta(p)}} \sum_{j=0}^{\beta(p)-1} \pt{\M_{-\tau/\beta(p)}  \T_{-\beta(p)-j} \delta_0 }\pt{\xi}\\
 \nonumber = &\frac{1}{\sqrt{\beta(p)}} \sum_{j=0}^{\beta(p)-1} e^{- \pii \frac{\tau}{\beta(p)} \xi }\, \delta_0\pt{\xi-\beta(p)-j}\\
 = &\frac{1}{\sqrt{\beta(p)}} \sum_{j=0}^{\beta(p)-1} e^{ - \pii  \pt{\beta(p)+j} \frac{\tau}{\beta(p)}  }\delta_{\beta(p)+j}(\xi).
\end{align}
Let us compute the $\ststst_\phi$-transform of a basis function $\dnbt$ 
with a general window $\phi$ belonging to $Z$.
We assume $\widehat{\phi}$ continuous.
By Theorem \ref{thm:contl2} and equation \eqref{eq:foudos}, we obtain
\begin{align}
\label{eq:sD}
\nonumber \phantom{=} &e^{ \pii b \xi }\pt{\ststst_\phi{\dnbt }}\pt{b, \xi}\\
\nonumber = & \Fi_{\zeta\mapsto b} \left( \overline{\widehat{\phi} \left( \frac{\zeta-\xi}{\xi}\right)}
\pt{\F\dnbt} (\zeta)\right)\pt{b}\\
\nonumber= &\F^{-1}_{\zeta\mapsto b}\left(\sum_{j=0}^{\beta(p)-1} 
\frac{e^{-\pii (\beta(p)+j)\tau/\beta(p)}} {\sqrt{\beta(p)}}
\overline{\widehat{\phi} \left( \frac{\zeta-\xi}{\xi}\right)} \delta_{(\beta(p)+j)}(\zeta)\right)\pt{b} \\
\nonumber= &\F^{-1}_{\zeta\mapsto b}\left(\sum_{j=0}^{\beta(p)-1} \frac{e^{-\pii (\beta(p)+j)\tau/\beta(p)}}
{\sqrt{\beta(p)}}
\overline{\widehat{\phi} \left( \frac{\beta(p) +j -\xi}{\xi}\right) }\delta_{(\beta(p)+j)}(\zeta)\right)\pt{b} \\
= &\left(\sum_{j=0}^{\beta(p)-1} \frac{e^{- \pii (\beta(p)+j)\tau/\beta(p)}} {\sqrt{\beta(p)}}
\overline{\widehat{\phi} \left( \frac{\beta(p) +j -\xi}{\xi}\right)} e^{ \pii b(\beta(p)+j)}\right)\pt{b}.
\end{align}
We set, for each fixed window $\phi$, $\cpfj : \Rr\to \Rr$ as 
\begin{align}
\label{eq:cjk}
\cpfj (\xi) = \overline{ \widehat{\phi}\left( \frac{\beta(p)+ j-\xi }{\xi}\right)}, \quad \xi\neq 0.
\end{align}
Hence, \eqref{eq:sD} simplifies into 
\begin{align}
\label{eq:sDc}
\pt{\ststst_\phi{\dnbt}}\bxi = e^{- \pii b \xi }\left(\sum_{j=0}^{\beta(p)-1} \frac{e^{\pii (\beta(p)+j) (b-\tau/\beta(p))}} {\sqrt{\beta(p)}}
\cpfj \pt{\xi}\right), \quad \xi\neq 0.
\end{align}
Equations \eqref{eq:sDc} and \eqref{eq:dost}
 provide an explicit expression of the $\ststst_\phi$-transform of 
a periodic signal $f$ in terms of its Stockwell coefficients $\fdnbt$.
Notice that if $\hat{\phi}$ is not a continuous function then 
equation \eqref{eq:sD}, \eqref{eq:cjk} and equation \eqref{eq:sDc}
must be understood as \emph{a.e.} equivalences.
\section{Discretization of the \texorpdfstring{$\ststst_\phi$-transform}{S-transform} }
\label{sec:discr}
Let us consider an admissible window and a dyadic decomposition of the frequency domain (see 
Section \ref{sec:propbase}).
We study the $\ststst_\phi$-transform of the periodic signal $f$ at 
$\xi= \nu(p)$.
Some conditions on the window $\phi$ are necessary in order to evaluate 
 $\ststst_{\phi}$-transform punctually.
 \begin{Ass}
  \label{Ass:wind}
  Let $\widehat{\phi}$ be a
  function in $L^\infty(\Rr)$ such that  $\widehat{\phi}|_{\pt{-\frac{1}{3}, \frac{1}{3}}}$
  is continuous, and such 
 that
 \begin{align*}
    \widehat{\phi}\pt{\xi}&\neq 0, \quad |\xi|<\frac{1}{3},  \\
    \widehat{\phi}\pt{\xi}&=0, \quad |\xi|>\frac{1}{3}, \qquad a.e.  \\
    \nonumber
     \lim_{\xi\to -\frac{1}{3}^+}& \widehat{\phi}\pt{\xi}=c<\infty.
 \end{align*}
 Notice that  $\phi$ belongs to the set of admissible windows $Z$.
 \end{Ass}

In the sequel we want  to evaluate
$\widehat{\phi}$ punctually. So, we need to perform a regularizing procedure.

\begin{lem}
  \label{lem:seq}
 Let $\phi$ be an admissible function satisfying Assumption \ref{Ass:wind}.
 Then it is possible to construct a  sequence of continuous functions 
 $\ptg{\phi_R}_{R=1}^{\infty}$  such that
 \footnote{Notice that, if $\phi$ satisfies \eqref{eq:condsup} and \eqref{eq:compat}, 
 then \eqref{eq:conpunt} implies \eqref{eq:connorm} by means of Lebesgue's Convergence Theorem.}
 \begin{align}
   \nonumber
   \widehat{\phi}_R& \mbox{ is continuous },\\
   \label{eq:conpunt}\widehat{\phi}_R\pt{\xi} &\to \widehat{\phi}\pt{\xi}, \quad \text{punctually},\\
   \label{eq:connorm}\phi_R\pt{\xi} &\to \phi\pt{\xi}, \quad \text{ in the set of admissible windows}\; Z.
 \end{align}
 Moreover, we can suppose that 
 \begin{align}
 \label{eq:condsup}
 \overline{\widehat{\phi}_{\Raggio}\pt{\xi}} & =  0,
 \qquad \xi\in \Rr\setminus \left({-\frac{1}{3}-\frac{2}{3 \beta( \Raggio)}}, {\frac{1}{3}}\right),\\ 
 \label{eq:compat}
 \overline{\widehat{\phi}_{\Raggio}\pt{\xi}}&=\widehat{\phi}\pt{\xi}, \quad \xi \in \pt{ -\frac{1}{3},{\frac{1}{3}-\frac{2}{3\beta \pt{\Raggio}}}}.
 \end{align}
\end{lem}
\begin{proof}
 We can consider the smooth function
 \[
  \omega_R\pt{\xi}= \left\{ \begin{array}{ll}
                           0 &, \quad \xi \in \R \setminus \pt{-\frac{1}{3}-\frac{2}{3 \beta\pt{\Raggio}}, \frac{1}{3} },\\  
                           1 &, \quad \xi \in \pt{-\frac{1}{3}, 
                           \frac{1}{3}-\frac{2}{3 \beta(\Raggio)} } . 
                            \end{array}
		   \right. 
 \]
 Since $\phi$ satisfies Assumption \ref{Ass:wind}, we can define
 \[
  \widetilde{\widehat{\phi}}\pt{\xi}=\left\{
				   \begin{array}{cc} 
				      \lim_{\xi\to -\frac{1}{3}^+}\widehat{\phi}\pt{\xi}& \quad \xi\leq -\frac{1}{3},\\
	                               \widehat{\phi}\pt{\xi}, &\quad \xi>-\frac{1}{3}.		   
				   \end{array}
			    \right. .
 \]
 Then $\phi_\Raggio(t)=\F_{\xi \mapsto t}^{-1}\pt{\omega_R\pt{\xi} \widetilde{\widehat{\phi}}\pt{\xi}}$
 has the desired properties.
\end{proof}

Let $\phi$ be an admissible window satisfying Assumption \ref{Ass:wind}
and $\ptg{\phi_\Raggio}_{\Raggio=1}^{\infty}$ a sequence as in Lemma \ref{lem:seq}. Then, by 
equation \eqref{eq:sDc}, we can write
\begin{align}
 \label{eq:sDd}
\pt{\ststst_{\phi_\Raggio} \dnbti }\pt{b, \nu(p)}=
e^{ -2\pi i b \nu(p) }\left(\sum_{j=0}^{\beta(p')-1} \frac{e^{2\pi i(\beta(p')+j) (t-\tau/\beta(p'))}} {\sqrt{\beta(p')}}
c_{p',j}^{\phi_{\Raggio}} (\nu(p))\right). 
\end{align}
Clearly, it is crucial to understand the values $c_{p',j}^{\phi_{\Raggio}}(\nu(p))$, 
which depend on the window $\phi$ only if $|p|\leq \Raggio$ and $|p'|\leq R$. 

\begin{pipi}
\label{prop:pnotp}
Let $\phi$ be an admissible window satisfying Assumption \ref{Ass:wind}
and $\ptg{\phi_{\Raggio}}_{R=1}^{\infty}$ be the associated sequence
defined in Lemma \ref{lem:seq}. Then
  \begin{align}
   \label{eq:zero}
   c_{p',j}^{\phi_{\Raggio }}( \nu(p))=0, 
   \qquad \forall j=0, \ldots, \beta(p')-1 \quad \mbox{if }p'\neq p,
   \;|p|\leq \Raggio, |p'|\leq \Raggio.
  \end{align}
\end{pipi}

\begin{proof}
 We restrict ourselves to positive $p'$. For $p'<0$, it suffices to consider the adjoint.
 
 Let $|p|<\Raggio$, as in \eqref{eq:zero}. In view of the properties of $\phi_R$,
 in particular \eqref{eq:condsup}, the condition
 \begin{multline}
  \label{eq:condsup1}
  \left( \frac {\beta(p') +j}{\nu(p)} \right)-1 \not \in \left(
 { -\frac{1}{3}-\frac{2}{3 \beta\pt{\Raggio } } } ,
 { \frac{1}{3} }\right), \\ \quad  p\neq p',  j=0, \ldots, \beta(p')-1
 \end{multline}
implies relation \eqref{eq:zero}.
 If $p$ is negative, then $\nu(p)<0$ and 
\begin{align*}
\left( \frac {\beta(p') +j}{\nu(p)} \right)-1 <- 1\leq -\frac{1}{3}-\frac{2}{3 \beta(R)},
\end{align*}
 hence \eqref{eq:condsup1} is fulfilled for all $j=0, \ldots, \beta(p')-1$.

 If $p$ positive, 
 recalling the definition of $\beta(p')$ and $\nu(p')$,  condition \eqref{eq:condsup1}
 turn into
  \begin{multline}
   \label{eq:supg}
 \frac{2}{3}\left( \frac {\beta(p')}{\beta(p) }+\frac{j}{ \beta(p)} \right)-1 
 \not \in \left(
 { -\frac{1}{3}-\frac{2}{3 \beta\pt{\Raggio } } } ,
 { \frac{1}{3} }\right), \\ \quad   j=0, \ldots, \beta(p')-1, \; p\neq p'.
  \end{multline}

 If $p\neq p'$, then we have to consider two cases.
\begin{description}
\item [Case I - ] $p'< p$.
\end{description}
The definition of $\beta(p')$ implies that $\beta(p')\leq \beta(p)/2$. Therefore,
\begin{align*}
 &\frac{2}{3}\left( \frac {\beta(p')}{\beta(p) }+\frac{j}{ \beta(p)} \right) -1 \leq 
 \frac{2}{3}\left( \frac {1}{2 }+\frac{j}{ \beta(p)} \right) -1\\
 &\leq - \frac{2}{3} +\frac{2}{3}\frac{j}{\beta(p)}\leq  -\frac{2}{3}+\frac{2}{3} \frac{\beta(p')-1}{\beta(p)} 
 \leq  -\frac{2}{3}+\frac{1}{3}- \frac{2}{3\beta(p)}\leq - \frac{1}{3}-\frac{2}{3 \beta(\Raggio)}. 
\end{align*}
\begin{description}
\item [Case II - ] $p'> p$.
\end{description}
We have $\beta(p)\leq \beta(p')/2$, so we can write
\begin{align*}
 \frac{2}{3}\left( \frac {\beta(p')}{\beta(p) }+\frac{j}{ \beta(p)} \right)-1  \geq 
 \frac{2}{3} \left( 2 +\frac{ j}{ \beta(p)} \right) -1 \geq \frac{1}{3} + \frac{2}{3}\frac{j}{\beta(p)} \geq   \frac{1}{3}. 
\end{align*}
Thus, \eqref{eq:supg} is fulfilled in both cases.
\end{proof}

Let $\phi$ be an admissible window satisfying Assumption \ref{Ass:wind} and 
$\ptg{\phi_R}$ be as in Lemma \ref{lem:seq}. Then, by Proposition \ref{prop:pnotp},
the expression \eqref{eq:sDd} assumes a simplified form since it vanishes for 
all $p'\neq p$, provided $|p'|\leq \Raggio$ and $|p|\leq \Raggio$. When $p=p'$ we have
\begin{align}
\label{eq:sconbase}
\pt{\ststst_{\phi_\Raggio}{\dnbt}}\pt{b, \nu(p)} = e^{ -\pii b \nu(p) }\left(\sum_{j=0}^{\beta(p)-1} \frac{e^{\pii (\beta(p)+j) (b-\tau/\beta(p))}} {\sqrt{\beta(p)}}
c_{p,j}^{\phi_{\Raggio}} (\nu(p))\right). 
\end{align}
Assume that $ c_{p,j}^{\phi_{\Raggio}} \pt{\nu\pt{p}}=1$ for all $j=0,\dots ,\beta\pt{p}-1$, then, via \eqref{eq:sconbase} 
\begin{align}
\label{eq:sbaseubase}
\pt{\ststst_{\phi_{\Raggio}}{\dnbt}}\pt{b, \nu(p)}= e^{ -\pii b \nu(p) }\dnbt(b). 
\end{align}

In order to extend \eqref{eq:sbaseubase} to all $\dnbt$, we introduce the following proposition. 

\begin{pipi}\label{pro:chibas}
Set $\cchi=\F^{-1}\chi$ be such that
 \begin{align}
 \label{eq:phiR}
 \overline{\pt{\F \cchi}}\pt{\xi}=\overline{\chi} \pt{\xi}=\left\{
                                            \begin{array}{ll} 
  					 0 &\quad  \xi \in \left( -\infty, -\frac{1}{3} \right]
  					 \cup \left[ \frac{1}{3}, +\infty \right)\\ 					 
  					 1 & \quad \xi \in \left(-\frac{1}{3},\frac{1}{3} 
  					 \right)
                                       \end{array}
\right..
 \end{align} 
Then $\cchi$ satisfies Assumption \ref{Ass:wind} and
\begin{align}
\label{eq:punctbas}
\pt{\ststst_{\cchi_\Raggio}{\dnbt}}\pt{b, \nu(p')} = 
e^{ -\pii b \nu(p) }\dnbt(b) \delta_0\pt{p-p'},\quad \text{for all } |p|\leq \Raggio, |p'|\leq \Raggio,
\end{align}
where $\ptg{\cchi_{\Raggio}}_{R=1}^{\infty}$ is a sequence converging to $\cchi$ as
in Lemma \ref{lem:seq}.
\end{pipi}

\begin{proof}
It follows from the definition of $\cpfj$ and by \eqref{eq:sbaseubase}.
\end{proof}
In order to extend the punctual evaluation \eqref{eq:punctbas} to all periodic signal
in $\Lloi$ we need to introduce another regularizing procedure in the frequency domain.
\begin{Def}
  \label{def:cut}
 We define the Fourier multiplier
 \begin{align}
  T_\Raggio: \Lloi &\to \Lloi\\
      f=\sum_{k\in \Z}\hat{f}\pt{k} e^{\pii k t}&\mapsto 
      \sum_{|k|< 2 \beta\pt{\Raggio}} \hat{f}\pt{k} e^{\pii k t}.
 \end{align}

\end{Def}

\begin{pipi}
\label{prop:DOST}
Let $f$ be a periodic signal and $\ptg{\cchi_\Raggio}_{R=1}^{\infty}$ defined as in 
Proposition \ref{pro:chibas} and $T_R$ as in Definition \ref{def:cut}. 
Then
\begin{align}
\label{eq:dostloc}
  \pt{\ststst_{\cchi_{\Raggio}} T_R f}\pt{\frac{\tau}{\beta\pt{p}}, \nu(p)} = 
  (-1)^\tau \sqrt{\beta(p)} \fdnbt ,\quad \tau=0, \ldots, \beta(p)-1,\quad |p|\leq \Raggio,
\end{align}
where
\[
  \fdnbt= \pt{f, \dnbt}_{\Lloi}.
\]

\end{pipi}

\begin{proof}
Since the functions $\pt{\dnbt}$ form an orthonormal basis of $\Lloi$, we have
\[
  f(t)= \sum_{p', \tau'} \pt{f, D_{p',\tau'}}_{\Lloi} D_{p',\tau'}(t), \quad a.e..
\]
Notice that
\[
  \pt{T_{R}f}(t)= \sum_{|p'|\leq R} \sum_{\tau'=0}^{\beta(p')-1}f_{p',\tau'} D_{p',\tau'}(t), 
\]
where $f_{p',\tau'}=\pt{f, D_{p',\tau'}}_{\Lloi}$.
By linearity, 
\begin{align*}
  \pt{\ststst_{\cchi_{\Raggio}} T_R f}\pt{\frac{\tau}{\beta\pt{p}}, \nu(p)}=
  \sum_{|p'|\leq R} \sum_{\tau'=0}^{\beta(p')-1} f_{p',\tau'} \pt{\ststst_{\cchi_{\Raggio}} D_{p',\tau'}}\pt{\frac{\tau}{\beta\pt{p}}, \nu(p)}.
\end{align*}
If $|p|\leq \Raggio$, by Proposition \ref{pro:chibas}, 
\begin{align}
  \label{eq:quasi}
  \pt{\ststst_{\cchi_{\Raggio}} T_R f}\pt{\frac{\tau}{\beta\pt{p}}, \nu(p)}=
  \sum_{\tau'=0}^{\beta(p)-1}f_{p,\tau'} \pt{\ststst_{\cchi_{\Raggio}} D_{p,\tau'}}
  \pt{\frac{\tau}{\beta\pt{p}}, \nu(p)}.
\end{align}
In Corollary \ref{pro:nul} we proved that 
\begin{align*}
D_{p, \tau'}\left( \frac{\tau}{\beta(p)}\right)= \sqrt{\beta(p)}\delta_0(\tau-\tau').
\end{align*}
Therefore, \eqref{eq:quasi} turns into
\begin{align*}
  \pt{\ststst_{\cchi_{\Raggio}} T_Rf}\pt{\frac{\tau}{\beta(p)}, \nu(p)}& = e^{- \pii \nu(p) \frac{\tau}{\beta(p)}}\fdnbt\,\dnbt\pt{\frac{\tau}{\beta(p)}}\\
  &= e^{- \pii \nu(p) \tau/\beta(p)} \sqrt{\beta(p)}\fdnbt.
\end{align*}
Since $\nu(p)= \pm 3/2 \,\beta(p)$,
\begin{align*}
e^{- \pii \nu(p) \tau/\beta(p)}= e^{\mp 3 \pi \mathrm{i}\,  \tau }= (-1)^\tau.
\end{align*}
Therefore, finally, 
\begin{align}
\label{eq:stcoeff}
\pt{\ststst_{\cchi_\Raggio} T_Rf}\pt{\frac{\tau}{\beta(p)}, \nu(p)}= (-1)^\tau \sqrt{\beta(p)} \fdnbt, \quad |p|\leq \Raggio. 
\end{align}
\end{proof}

%
The definition of $\cchi_{\Raggio}$ in \eqref{eq:phiR} 
implies that 
\[
 \norm{\cchi_R- \cchi }_Z\to 0.
\]
Moreover, it is immediate that,  for all $f\in \Lloi$, $\norm{T_R f- f}_{\Lloi}\to 0$.
Therefore, by the continuity properties of $S$, proven in Theorem \ref{thm:contl2},
 for all $f\in \Lloi$
\begin{align}
 \label{eq:limnorm}
 \norm{\pt{S_{\cchi_{\Raggio}} T_R f} - \pt{S_{\cchi} f}}_Y\to 0, \quad \Raggio\to \infty.
\end{align}

%
%
\noindent Equations \eqref{eq:stcoeff} and \eqref{eq:limnorm} clarify the representation of the $\ststst$-transform of a periodic signal $f$ via the Stockwell coefficients $\fdnbt$. 
Moreover, \eqref{eq:stcoeff} explains the role of the multiplicative factor $(-1)^\tau$
in front of the basis functions $\dnbt$ used by R. G. Stockwell in \cite{ST07}.

\begin{rem}
 In the paper we have always considered a symmetric partition of the frequency from the positive and negative side.
 Actually, the algorithm is slightly different: see \cite{WAth,WO08,WO09} for details.
\end{rem}

\section{Window Adapted Basis Construction}
\label{sec:baseephi}

\begin{figure}[t]
\centering
\includegraphics[width=.8\textwidth]{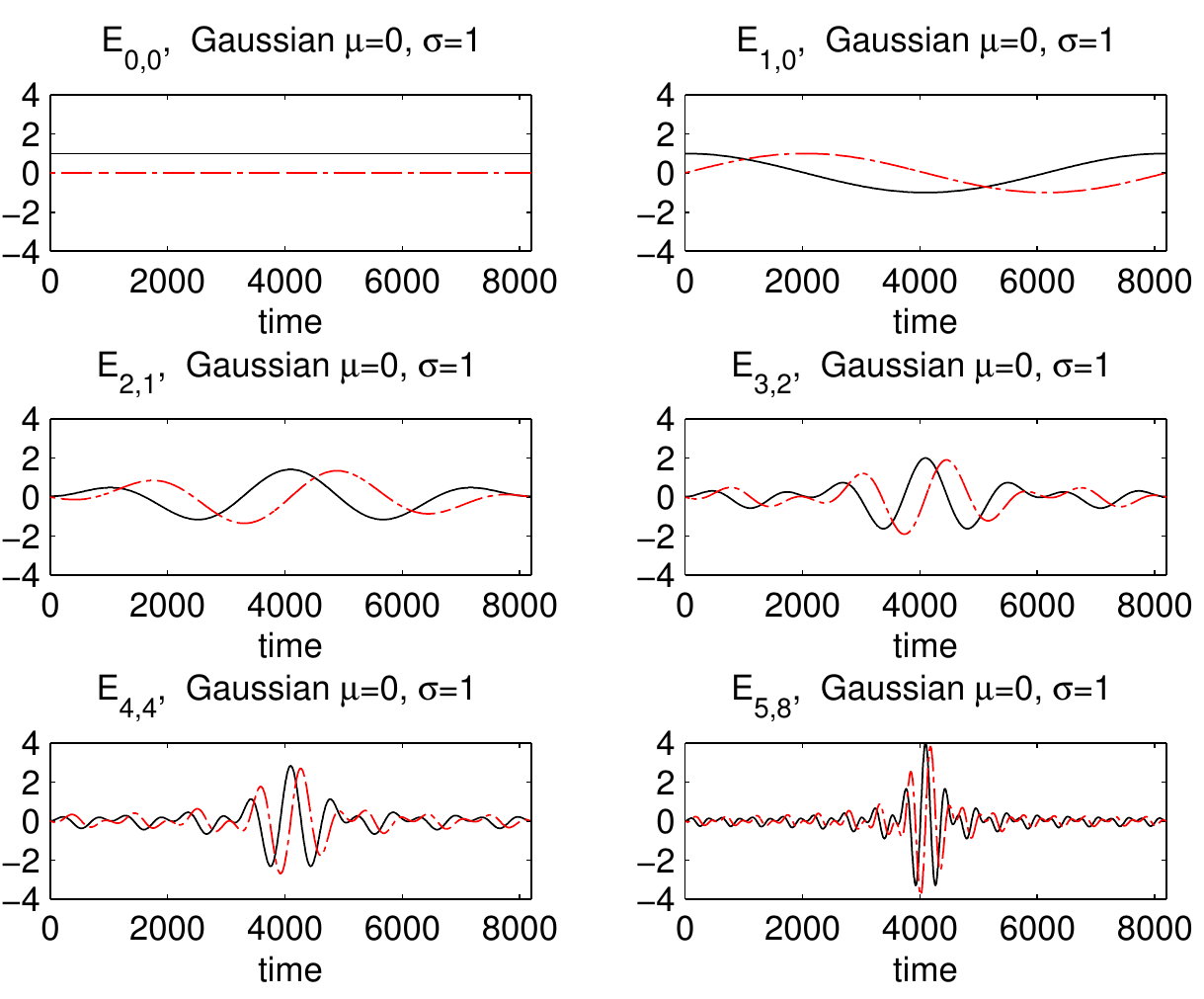}
\caption{$\ephipt$ basis functions in increasing frequency $p$-bands. Black line = real, red line = imaginary. 
$\widehat{\phi}$ is a truncated Gaussian window with $\mu=0 $ and $\sigma =1 $. Notice the similarities with Figure 1. Indeed, in this case the ratio 
$\pt{\delta/M}^2$  
is approximately $0.8836$
and $\pt{M/\delta}^2$ is approximately $1.13173$, see Theorem \ref{thm:ultimo}.}
\label{fig:04}
\end{figure}

In this section we  determine a basis of $L^2([0,1])$ adapted to an admissible window $\phi$
satisfying Assumption \ref{Ass:wind}.
As explained in the introduction, we want to find a basis $E^\phi_p$
such that $\ststst_\phi E^\phi_p$ is \emph{local}
both in time and in frequency and such that the evaluations of 
all coefficients $f^\phi_p=(f, E^\phi_p)_{\Lloi}$ is fast -- $\mathcal{O}\pt{N \log N}$. 
In Section \ref{sec:propbase}, we proved that $\dnbt$ is a basis of 
$\Lloi$ which is \emph{local} both in time and in frequency.
Moreover, in Section \ref{sec:discr}, we showed that the natural 
discretization of the time-frequency domain 
in this setting is given by the dyadic decomposition in 
the frequency domain and the $\tau/\beta(p)$ grid in the time domain.
So, it is natural to change our task in finding a basis $\ephipt$ such that
\begin{align}
\label{eq:condnu}
 \pt{\ststst_\phi \ephipt }\pt{b, \nu(p)} &= e^{- \pii b \nu(p)} \dnbt(b).
\end{align}
As in the previous section, in order to obtain 
the punctual evaluation \eqref{eq:condnu}, we introduce
a sequence $\ptg{\phi_\Raggio}_{\Raggio=1}^{\infty}$ as in Lemma \ref{lem:seq}.
In order to keep the notation easier, we set
\[
  c^{\phi}_{p,j}\pt{\nu(p)}=c^{\phi_\Raggio}_{p,j}\pt{\nu(p)}, \quad  |p|\leq \Raggio.
\]
Notice that this definition makes sense in view of \eqref{eq:compat}.
\begin{thm}\label{th:basevariabile}
 Let $\phi$ be an admissible window satisfying Assumption \ref{Ass:wind} and
 \begin{align}
  \label{eq:defbasevar}
  \ephipt \pt{t} &= \frac{1}{\sqrt{\beta(p)}} \sum_{j=0}^{\beta(p)-1} \left[c^\phi_{p,j}(\nu\pt{p})\right]^{-1} e^{\pii \pt{\beta(p)+j}\pt{t- \frac{\tau}{\beta(p)}}}.
 \end{align}
Then
\begin{equation}
\label{eq:basevar}
 \pt{\ststst_{\phi_\Raggio} \ephipt}(b, \nu(p))= e^{- \pii  b \nu(p)}  \dnbt(b), \quad |p|\leq \Raggio.
\end{equation}
Moreover, 
\[
 \bigcup_{p \in \Z} E^\phi_p,
\]
where
\[
 E_p^\phi= \ptg{E^\phi_{p, \tau}}_{\tau=0, \ldots, \beta(|p|)-1}
\]
is a basis of $\Lloi$.
\end{thm}

\begin{figure}[t]
\centering
\includegraphics[width=.8\textwidth]{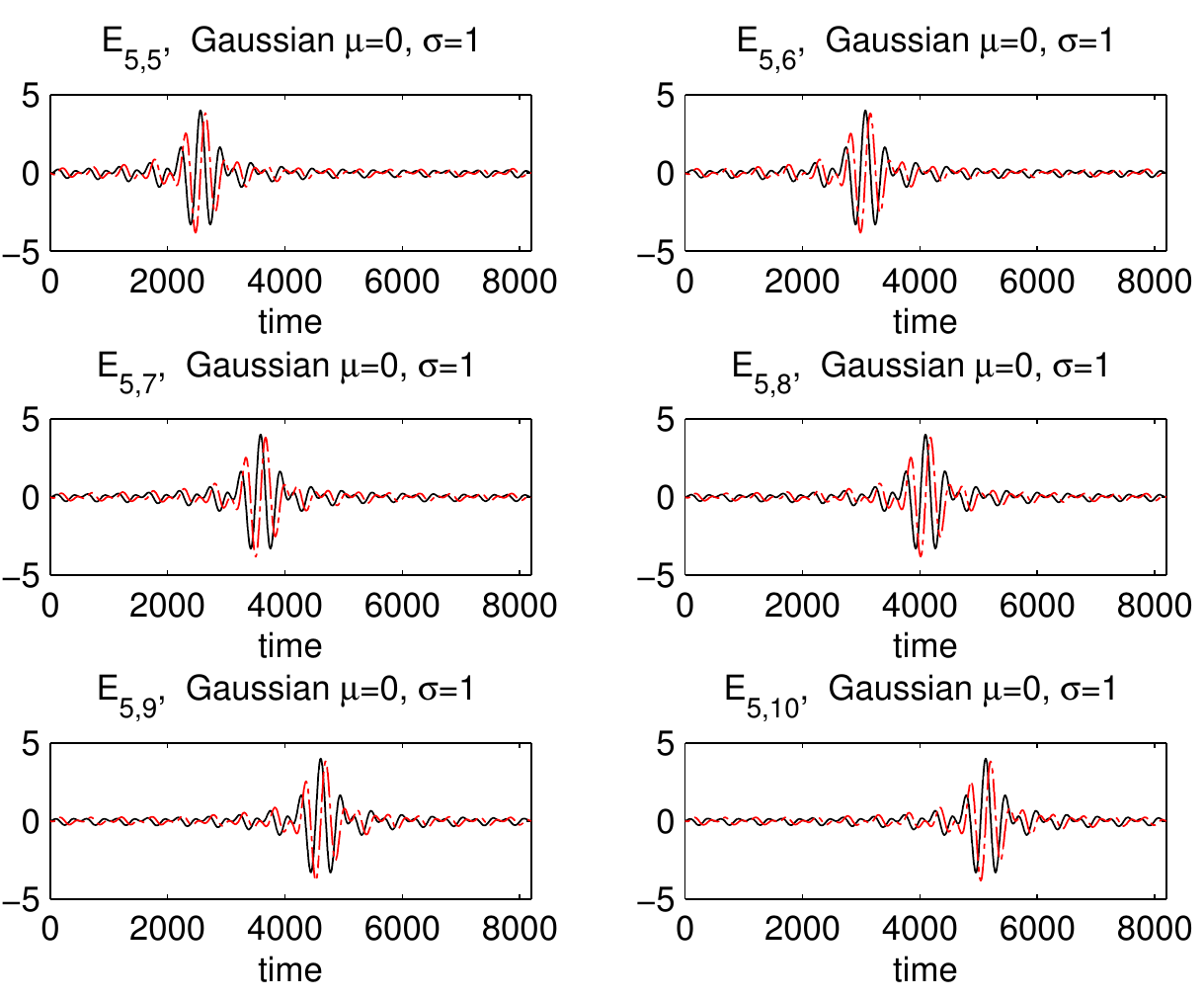}
\caption{$\ephipt$ basis functions in the same frequency $p$-band ($p=5$). Black line = real, red line = imaginary. 
$\widehat{\phi}$ is a truncated Gaussian window with $\mu=0 $ and $\sigma =1 $. See Figure 2 for comparison. }\label{fig:05}
\end{figure}

\begin{rem}
\label{rem:dost}
  Take $\cchi$ as in \eqref{eq:phiR}, then 
  \begin{align*}
  c^{\cchi}_{p,j} \pt{\nu\pt{p}} = 1,\qquad 
  \end{align*}
  for all $p$ and $j$.
  So, 
  by \eqref{eq:defbasevar} and \eqref{eq:Dpriscritta}, 
  \begin{align*}
    \ecchipt \pt{t} =& \frac{1}{\sqrt{\beta(p)}} \sum_{j=0}^{\beta(p)-1} \left[c^{\cchi}_{p,j}(\nu\pt{p})\right]^{-1} e^{\pii \pt{\beta(p)+j}\pt{t- \frac{\tau}{\beta(p)}}}\\
    = & \frac{1}{\sqrt{\beta(p)}} \sum_{j=0}^{\beta(p)-1} e^{\pii \pt{\beta(p)+j}\pt{t- \frac{\tau}{\beta(p)}}}
    \\ 
    = & \dnbt (t).
  \end{align*}
  Hence, the functions $E_{p,\tau}^\phi$ are a proper generalization of the DOST functions.
\end{rem}

\begin{figure}[t]
\centering
\includegraphics[width=.8\textwidth]{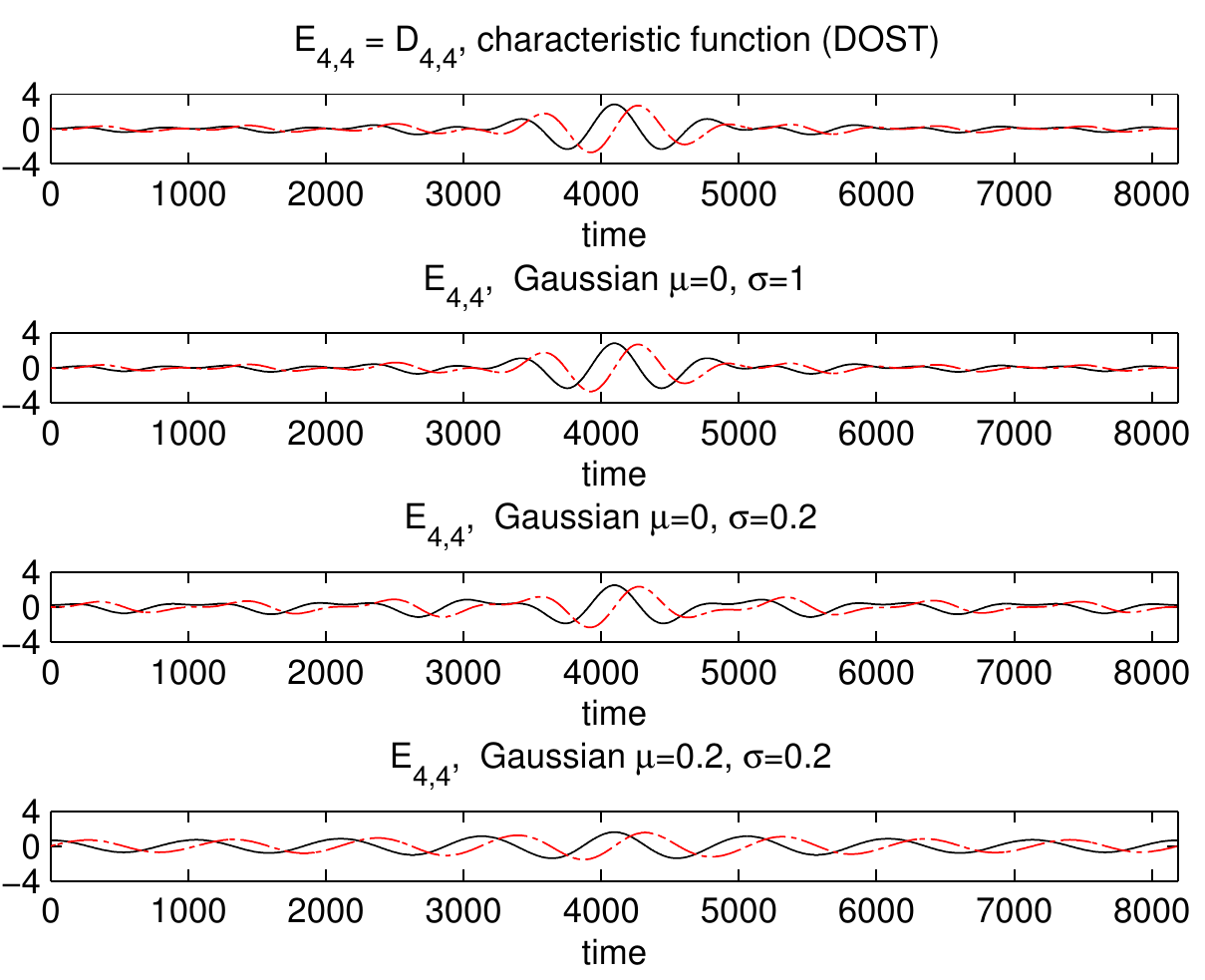}
\caption{$\ephipt$ basis functions with $p=4$ and $\tau = 4$ with different windows. 
Black line = real, red line = imaginary. 
The Fourier transform of $\widehat{\phi}$ is $\chi_{(-1/3,1/3)}$ in the first plot, then a truncated Gaussian with varying $\mu$ and $\sigma$.}
\label{fig:07}
\end{figure}

\begin{proof}
By equation \eqref{eq:sDc}, it follows that the functions $\ephipt$ do satisfy \eqref{eq:basevar}.
 So, we only need to prove that  $\bigcup_{p\in \Z} E^\phi_p$ is a basis of $\Lloi$. 
 
 Notice that
 \begin{align*}
  E^\phi_p &\subseteq \Span \ptg{ e^{\pii  k t}}_{k \in \ptq{\beta(p),2 \beta(p)-1}} = \Span \ptg{ \dnbt }_{\tau=0, \ldots, \beta(p)-1}.
 \end{align*}
It is sufficient to check that $E^\phi_p$ is a linear independent set.
Let us assume that there exist $\{\alpha_\tau\}_{\tau=0}^{\beta(p)-1}$ such that
\begin{align*}
 \sum_{\tau=0}^{\beta(p)-1} \alpha_\tau \ephipt \pt{t} =0.
\end{align*}
Then, by \eqref{eq:basevar}, for $|p| \leq \Raggio $, we obtain
\begin{align*}
0  &=\pt{\ststst_{\phi_{\Raggio}}  \sum_{\tau=0}^{\beta(p)-1} \alpha_\tau \ephipt}\pt{b, \nu(p)}\\
   &= \sum_{\tau=0}^{\beta(p)-1} \alpha_\tau \pt{\ststst_{\phi_\Raggio} \ephipt }\pt{b, \nu(p)}\\
   &= e^{-\pii b \nu\pt{p}} \sum_{\tau=0}^{\beta(p)-1} \alpha_\tau \,  \dnbt (b).
\end{align*}
Hence,
\begin{align}
\label{eq:contradiction}
   \sum_{\tau=0}^{\beta(p)-1} \alpha_\tau \dnbt(b) = 0 .
\end{align}
Since $\dnbt$ is a basis, \eqref{eq:contradiction} implies that $\alpha_\tau$ are all zeros. That is, $\ephipt$ are linear independent.
\end{proof}

\begin{pipi}
\label{prop:comput}
Let $\ephipt$ as in Theorem \ref{th:basevariabile} and let $f$ be a finite signal. 
Then the evaluation of the coefficients
\begin{align*}
 f^\phi_{p,\tau} = \pt{f, \ephipt }_{\Lloi}
\end{align*}
has computational complexity $\mathcal{O}(N \log N)$, where $N$ is the length of $f$.
\end{pipi}

\begin{proof}
By Plancharel's Theorem we can write
\[
 f^{\phi}_{p, \tau}= \pt{f, \ephipt}_{\Lloi}= \pt{\hat{f}, \widehat{\ephipt}}_{l^2(\Z)}.
\]
Using the explicit expression of the basis $\ephipt$, we obtain
\begin{align*}
 f^{\phi}_{p, \tau}&= \pt{\hat{f}, \frac{1}{\sqrt{\beta(p)}}\sum_{j=0}^{\beta(p)-1} \left[c^\phi_{p,j}(\nu(p))\right]^{-1} e^{- \pii  ( \beta(p)+j)(\tau/\beta(p)) }\delta_{\beta(p)+j}(\cdot)}_{l^2(\Z)} \\
 &= \frac{1}{\sqrt{\beta(p)}}\sum_{j=0}^{\beta(p)-1} \widehat{f}(\beta(p)+j) \overline{\left[c^\phi_{p,j}(\nu(p))\right]^{-1}} e^{ \pii  ( \beta(p)+j)(\tau/\beta(p)) } \\
 &=\pt{ R^\phi \hat{f}, \frac{1}{\sqrt{\beta(p)}}\sum_{j=0}^{\beta(p)-1}  e^{- \pii ( \beta(p)+j)(\tau/\beta(p)) }\delta_{\beta(p)+j}(\cdot)}_{l^2(\Z)} \\
 &=\pt{\F^{-1} R^\phi \hat{f}, \dnbt}_{\Lloi},
\end{align*}
where $R^{\phi}$ is a sequence in $\Z$ such that
\begin{equation}
 \label{eq:rphi}
 R^\phi(\beta(p)+j)= \overline{\left[c^\phi_{p,j}(\nu(p))\right]^{-1}}
\end{equation}
for all $p$ and related $j$.
Hence,  
\begin{equation}
\label{eq:equiv}
 f^\phi_{p, \tau}= \pt{f, \ephipt}_{\Lloi}= \pt{\tilde{f}, \dnbt}_{\Lloi}
\end{equation}
where $\tilde{f}= \F^{-1} R^{\phi} \hat{f}$.
Given $\tilde{f}$, computing \eqref{eq:equiv} using the  FDOST-algorithm introduced in \cite{WO09} has complexity $\mathcal{O}\pt{N \log N}$ and computing $\tilde{f}$ via FFT has complexity $\mathcal{O}\pt{N \log N}$. So, the computational complexity remains $\mathcal{O}\pt{N \log N}$. 
\end{proof}

\begin{figure}[t]
\centering
\includegraphics[width=.8\textwidth]{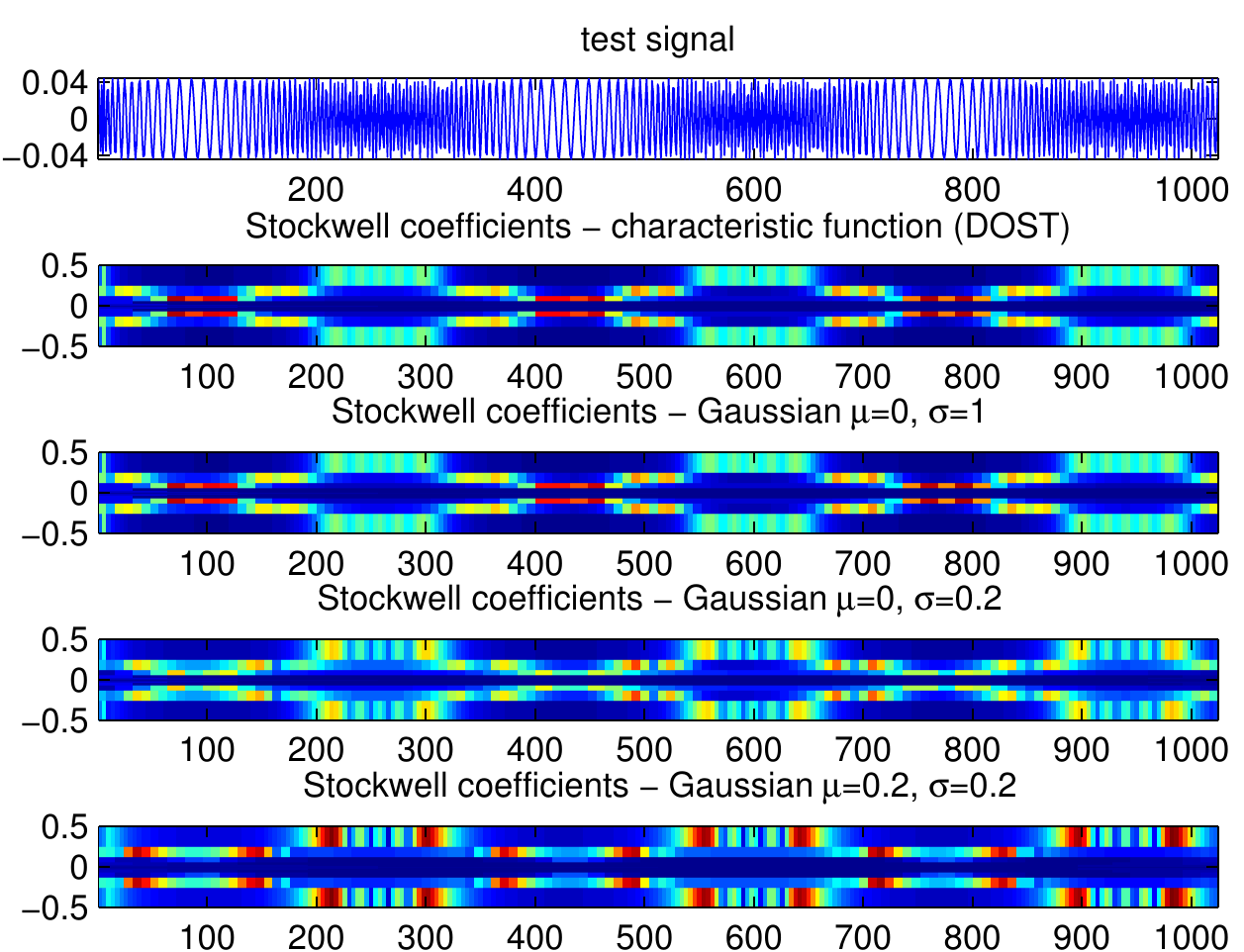}
\caption{Decompositions of a given test signal on different windowed basis.}
\label{fig:03}
\end{figure}

\begin{rem}
It is worth checking explicitly the computational complexity of the algorithm. To perform this task, we start evaluating the column vector $f^\phi_p$ given by 
\begin{align}
\nonumber
 f^\phi_p =& \{ f^\phi_{p, \tau}\}_{\tau=0}^{\beta(p)-1}\\
  \nonumber
  = &\ptg{\pt{f, \ephipt }_{\Lloi}}_{\tau=0}^{\beta(p)-1}\\
	\nonumber
	= &\ptg{\pt{\hat{f}, \widehat{\ephipt }}_{l^2(\Z)}}_{\tau=0}^{\beta(p)-1}\\
\nonumber =& \ptg{\pt{\hat{f},\frac{1}{\sqrt{\beta(p)}}
 \sum_{j=0}^{\beta(p)-1}\left[c^\phi_{p,j}(\nu(p))\right]^{-1} e^{- \pii  ( \beta(p)+j)(\tau/\beta(p)) }\delta_{\beta(p)+j}(\cdot
 )}_{l^2(\Z)}}_{\tau=0}^{\beta(p)-1} \\
 \nonumber
 =&\ptg{\frac{1}{\sqrt{\beta(p)}}\sum_{j=0}^{\beta(p)-1} \widehat{f}(\beta(p)+j)
 \overline{\left[c^\phi_{p,j}(\nu(p))\right]^{-1}} e^{ \pii  ( \beta(p)+j)(\tau/\beta(p)) }}_{\tau=0}^{\beta(p)-1} \\
 \nonumber
 =&\ptg{\frac{1}{\sqrt{\beta(p)}}\sum_{j=0}^{\beta(p)-1} \widehat{f}(\beta(p)+j)\overline{ 
 \left[c^\phi_{p,j}(\nu(p))\right]^{-1}} e^{ \pii  j (\tau/\beta(p)) }}_{\tau=0}^{\beta(p)-1} \\
 \nonumber
 =& \pt{\Fi_{j \mapsto \tau} \pt{ \pt{ R^\phi \; \widehat{f} }|_{\beta(p), \ldots, 2\beta(p)-1}(j) }}\pt{\tau}
\end{align}
where $R^\phi$ is defined as in \eqref{eq:rphi}. Therefore, first we have to perform 
the FFT of the signal $f$ ($\mathcal{O}\pt{N \log N}$),
and the multiplication by $R^\phi$ ($\mathcal{O}(N)$), then at each $p$ band we need to use the FFT to perform the 
anti  Fourier
transform with computational complexity $\mathcal{O}\pt{\beta(p) \log \beta(p) }$. Summing up the contribution 
of each $p$-band
we obtain the computational complexity of $\mathcal{O}(N \log N)$. 
\end{rem}

The basis $\ptg{\ephipt}_{p, \tau}$ is in general not orthogonal nor normal. Nevertheless, we can normalize it setting 
\begin{align}
\label{eq:ephifr}
{\Fphipt}(t)= \frac{\ephipt (t)}{\|\ephipt\|_{L^2([0,1])}},
\end{align}
so that
\begin{align*}
\norm{\Fphipt}_{L^2([0,1])}=1.
\end{align*}
Notice that 
\begin{equation}
\label{eq:normp}
\norm{\ephipt}_{\Lloi}=\norm{E^\phi_{p, \tau'}}_{\Lloi}=N^\phi_p 
\end{equation}
depends just on the $p$-band, not on $\tau$.

The basis $\ptg{\Fphipt(t)}_{p,\tau}$ fails in general to be orthogonal. 
Nevertheless, assuming a mild condition on $\phi$, we can prove that it is a frame.

\begin{thm}\label{thm:ultimo}
 Let $\phi$ be an admissible window function 
 satisfying Assumption \ref{Ass:wind}, and such that
 \begin{align}
 \label{eq:condlim}
  \inf_{\xi\in \pt{-1/3, 1/3}}& \abs{\overline{\widehat{\phi}(\xi)}} \geq \delta >0\\
  \label{eq:condsupp}
  \sup_{\xi\in \pt{-1/3, 1/3}}& \abs{ \overline{\widehat{\phi}(\xi)} }\leq M <\infty
 \end{align}
then the basis $\bigcup_{p \in \Z} F^\phi_p$ is a frame of $\Lloi$, where
\[
 F^\phi_p= \ptg{F^\phi_{p, \tau}}_{\tau=0, \ldots, \beta\pt{ \abs{p} }    }.
\]
In particular
\begin{align*}
 \pt{\frac{\delta}{M }}^2 \norm{f}^2_{\Lloi}\leq \sum_{p, \tau} \abs{\pt{f, F^\phi_{p, \tau}}_{\Lloi}}^2\leq  \pt{\frac{M}{\delta}}^2\norm{f}^2_{\Lloi}. 
\end{align*}
\end{thm}
\begin{proof}
 Notice that under the hypothesis \eqref{eq:condlim}, \eqref{eq:condsupp}, by \eqref{eq:normp}
\begin{equation}
 \label{eq:estnorm}
 \frac{1}{M} \leq N^\phi_p \leq \frac{1}{\delta}, \quad \forall{p \in \Z}.
 \end{equation}

 Observe, by a slight variation of \eqref{eq:equiv}, that
\[
 \pt{f,\Fphipt}_{\Lloi}=  \pt{\F^{-1} \widetilde{R^\phi}\hat{f}, \dnbt }_{\Lloi}
\]
where $\widetilde{R^\phi}$ is a sequence such that
\[
\widetilde{R^\phi}(\beta(p)+j)=\frac{R^\phi(\beta(p)+j)}{N^\phi_p} =\frac{\ptq{\overline{c_{p, j}^\phi(\nu(p))}}^{-1}} {N^\phi_p},
\]
 where $N^\phi_p$ is as in \eqref{eq:normp}.
 
If the window $\phi$ satisfies condition \eqref{eq:condsupp}, by \eqref{eq:estnorm}, we have
\begin{align}
\label{eq:normsup}
&\sup_{k \in \Z} \ptg{\abs{\widetilde{R^\phi}(k)}} \leq \frac{M}{\delta}<\infty,\\
\label{eq:norminfi}
&\inf_{k \in \Z} \ptg{\abs{\widetilde{R^\phi}(k)}} \geq \frac{\delta}{M}> 0.
\end{align}
Hence, since $\bigcup_{p \in \Z} D_p $ is an orthonormal basis and since $\F$ is a
unitary operator from $\Lloi$ to $l^2(\Z)$, we obtain
\begin{align*}
 \phantom{=} & \sum_{p, \tau}\abs{ \pt{f, \Fphipt}_{\Lloi} }^2 =\sum_{p, \tau} \abs{ \pt{\F^{-1} \widetilde{R^\phi} \hat{f}, \dnbt}_{\Lloi}}^2= \norm{\F^{-1}\widetilde{R^\phi} \hat{f}}^2_{\Lloi}\\
 = & \norm{\widetilde{R^\phi} \hat{f}}_{l^2(\Z)}  \leq \pt{\sup_{k \in \Z} 
 \ptg{\abs{\widetilde{R^\phi}(k)}}}^2  \norm{\hat{f}}^2_{l^2(\Z)} \leq \pt{\frac{M}{\delta}}^2 \norm{f}^2_{\Lloi},
\end{align*}
and
\begin{align*}
 \phantom{=} & \sum_{p, \tau}\abs{ \pt{f, \Fphipt}_{\Lloi} }^2 =\sum_{p, \tau} \abs{ \pt{\F^{-1} \widetilde{R^\phi} \hat{f}, \dnbt}_{\Lloi}}^2= \norm{\F^{-1}\widetilde{R^\phi} \hat{f}}^2_{\Lloi}\\
 = & \norm{\widetilde{R^\phi} \hat{f}}_{l^2(\Z)}  \geq \pt{\inf_{k \in \Z} \ptg{\abs{\widetilde{R^\phi}(k)}}}^2  \norm{\hat{f}}^2_{l^2(\Z)} \geq \pt{\frac{\delta}{M}}^2 \norm{f}^2_{\Lloi}.
\end{align*}
\end{proof}
Since  $\bigcup_{p \in \Z} F^\phi_p$ forms a frame, it is possible to
obtain abstractly the canonical dual frame, we denote $\widetilde{\Fphipt}$. 
So, following the same scheme of Proposition \ref{prop:DOST} and equation
\eqref{eq:basevar}, we have
\begin{align*}
  \pt{S_{\phi_{\Raggio}} T_R f}&\pt{\frac{\tau}{\beta(p)}, \nu(p)} \\
  &=\pt{S_{\phi}T_{\Raggio}\sum_{p',\tau'}\pt{f, \widetilde{\Fphiptp}}_{\Lloi} 
  \Fphiptp}\pt{\frac{\tau}{\beta(p)},\nu(p)}\\
                            &=\sum_{|p'|\leq \Raggio} \sum_{\tau'=0}^{\beta(p')-1} 
                            \pt{f, \widetilde{\Fphiptp}}_{\Lloi}
                            \pt{S_{\phi_\Raggio} 
                            \Fphiptp}\pt{\frac{\tau}{\beta(p)},\nu(p)}\\
                            &=\sum_{p'\leq\Raggio}\sum_{\tau'=0}^{\beta(p')-1} \pt{f, \widetilde{\Fphiptp}}_{\Lloi}                         
                            \frac{e^{-\pii \frac{\tau}{\beta(p)} \nu(p')}}{N_{p'}} 
                            \dnbtp\pt{\frac{\tau}{\beta(p)}, \nu(p)}\\
                            &=\pt{-1}^{\tau}\sqrt{\beta(p)} \frac{\pt{f, \widetilde{\Fphipt}}_{\Lloi}}{N_p}, \quad |p|\leq\Raggio .
\end{align*}

\begin{rem}
Notice that, when in equations \eqref{eq:condlim} and \eqref{eq:condsupp} $\delta=M$, we a have a tight-frame.
In the case of the DOST basis, $i.e.$  $\ecchipt = \dnbt$ it is clear that  $\delta=M=1$. So $\dnbt$ is a tight-frame. Actually, we have proven more: $\dnbt$ is an orthonormal basis.
\end{rem}
\section*{Acknowledgements}
The authors thank M. Berra, P. Boggiatto, E. Cordero, V. Giannini, M. Lupini, F. Nicola, A. Vignati, M. W. Wong and H. Zhu  
for fruitful discussions and comments.
We are grateful to the referees for a number of helpful 
suggestions for improvement in the article, in particular 
for remarks on Theorem \ref{thm:contl2} and Theorem \ref{thm:ultimo}.

The first author has been supported by the Gruppo  Nazionale per l'Analisi Matematica, la Probabilit\`a e 
le loro  Applicazioni (GNAMPA) of the Istituto Nazionale di Alta Matematica (INdAM) and by a Postdoc scholarship of the Universit\`a degli Studi di Torino.

The second author was supported by grants from the Doctoral School of Sciences and
Innovative Technologies, Ph.D. Program in Mathematics of Universit\`a degli Studi di Torino.

\bibliography{bibliodost2}
\bibliographystyle{abbrv}

\end{document}